\RequirePackage{fixltx2e}
\documentclass[english]{article}
\usepackage[T1]{fontenc}
\usepackage[latin9]{inputenc}
\usepackage{float}
\usepackage{dsfont}
\usepackage{amsmath}
\usepackage{amsthm}
\usepackage{amssymb}
\usepackage{geometry}
\usepackage{setspace}
\makeatletter
\theoremstyle{plain}
\newtheorem{thm}{\protect\theoremname}[section]
\theoremstyle{plain}
\newtheorem{cor}[thm]{\protect\corollaryname}
\theoremstyle{plain}
\newtheorem{prop}[thm]{\protect\propositionname}
\theoremstyle{remark}
\newtheorem{rem}[thm]{\protect\remarkname}
\theoremstyle{plain}
\newtheorem{lem}[thm]{\protect\lemmaname}
\theoremstyle{definition}
\newtheorem{defn}[thm]{\protect\definitionname}
\theoremstyle{definition}
\newtheorem{example}[thm]{\protect\examplename}

\@ifundefined{date}{}{\date{}}
\usepackage{graphicx}
\usepackage{fancyhdr}\usepackage{setspace}\usepackage{mathrsfs}\usepackage{enumerate}\usepackage{tikz}\usepackage{fixltx2e}\usepackage{amsthm}
\usepackage{bbm}
\AtBeginDocument{
  \renewbibmacro{in:}{}
}

\usetikzlibrary{decorations.markings}\usetikzlibrary{shapes.misc}
\usepackage[colorlinks=true,linkcolor=blue,citecolor=blue,urlcolor=blue]{hyperref}
\usepackage{cleveref}\usepackage{dsfont}\usetikzlibrary{calc}

\tikzset{->-/.style={decoration={   markings,   mark=at position .5 with {\arrow{>}}},postaction={decorate}}} 

\newcommand{\myoctagon}[1]{

      \coordinate (#1v1) at (0.3, 3);     \coordinate (#1u1) at (1.7,3);

    \coordinate (#1v2) at (2.2,2.65); \coordinate (#1u2) at (2.9,2.1);

    \coordinate (#1v3) at (3.2, 1.5); \coordinate (#1u3) at (3.2, 0.3);

    \coordinate (#1v4) at (2.9,-0.3); \coordinate(#1u4) at (2.2, -0.85);

    \coordinate (#1v5) at (1.7, -1.2); \coordinate (#1u5) at (0.3, -1.2);

    \coordinate (#1v6) at (-0.2,-0.85); \coordinate (#1u6) at (-0.8, -0.3);

    \coordinate (#1v7) at (-1.2,0.3); \coordinate (#1u7) at (-1.2, 1.5);

    \coordinate (#1v8) at (-0.9, 2.1); \coordinate (#1u8) at (-0.2, 2.65);

     \filldraw(#1v1) circle (2pt); \filldraw(#1u1) circle (2pt);

    \filldraw (#1v2) circle (2pt); \filldraw(#1u2) circle(2pt);

    \filldraw (#1v3) circle (2pt); \filldraw(#1u3) circle(2pt);

    \filldraw (#1v4) circle (2pt); \filldraw(#1u4) circle(2pt);

    \filldraw (#1v5) circle (2pt); \filldraw(#1u5) circle(2pt);

    \filldraw (#1v6) circle (2pt); \filldraw(#1u6) circle(2pt);

    \filldraw (#1v7) circle (2pt); \filldraw(#1u7)circle(2pt);

    \filldraw (#1v8) circle (2pt); \filldraw(#1u8) circle(2pt);

     \draw[->-] (#1v1)-- node[midway, above]{$a$}(#1u1);
\draw[->-] (#1v2)-- node[midway, above right]{$b$}(#1u2);
\draw[->-](#1u3)-- node[midway,right]{$a$}(#1v3);
\draw[->-](#1u4)-- node[midway, below right]{$b$}(#1v4);
\draw[->-](#1v5)--node[midway, below]{$c$}(#1u5);
\draw[->-](#1v6)-- node[midway, below left]{$d$}(#1u6);
\draw[->-](#1u7)-- node[midway, left]{$c$}(#1v7);
\draw[->-](#1u8)-- node[midway, above left]{$d$}(#1v8);

}

\newcommand{\gammaloop}{
\coordinate (1) at (0,3); \coordinate (2) at (2,3); \coordinate (3) at (3.4, 1.6); \coordinate(4) at (2, 0.2); \coordinate (5) at (0,0.2); \coordinate (6) at (-1.4, 1.6); 

\filldraw(1) circle(2pt); \filldraw(2) circle(2pt);\filldraw(3) circle(2pt);\filldraw(4) circle(2pt);\filldraw(5) circle(2pt);\filldraw(6) circle(2pt);
\draw[->-](1)--node[midway, above]{$a$}(2); \draw[->-](2)--node[midway, above right]{$a$}(3); \draw[->-](3)--node[midway, below right]{$c$}(4); \draw[->-](4)--node[midway, below]{$b$}(5); \draw[->-](6)--node[midway, below left]{$c$}(5); \draw[->-](1)--node[midway, above left]{$b$}(6);  
}

\makeatother

\usepackage{babel}
\usepackage[style=alphabetic,maxbibnames=99]{biblatex}
\providecommand{\corollaryname}{Corollary}
\providecommand{\definitionname}{Definition}
\providecommand{\examplename}{Example}
\providecommand{\lemmaname}{Lemma}
\providecommand{\propositionname}{Proposition}
\providecommand{\remarkname}{Remark}
\providecommand{\theoremname}{Theorem}

\begin{document}
\global\long\def\C{\mathbb{C}}%
\global\long\def\cnk{\left(\C^{n}\right)^{\otimes k}}%
\global\long\def\P{P_{k}(n)}%
\global\long\def\A{A_{k}(n)}%
\global\long\def\element{\left(e_{1}-e_{2}\right)\otimes\dots\otimes\left(e_{2k-1}-e_{2k}\right)}%
\global\long\def\Sn{S_{n}}%
\global\long\def\endsn{\mathrm{End}{}_{\Sn}}%
\global\long\def\xilambda{\sum_{\sigma\in S_{k}}\frac{d_{\lambda}}{k!}\chi^{\lambda}(\sigma)\left[v_{\sigma(1)}\otimes\dots\otimes v_{\sigma(k)}\right]}%
\global\long\def\p{\mathscr{P}}%
\global\long\def\dkn{\left\langle e_{i_{1}}\otimes\dots\otimes e_{i_{k}}:|\{i_{1},\dots,i_{k}\}|=k\right\rangle }%
\global\long\def\normxilambda{\xi_{\lambda}^{\mathrm{norm}}}%
\global\long\def\skdelta{S_{k}^{\Delta}}%
\global\long\def\vlambdacheck{\check{V}^{\lambda}}%
\global\long\def\Q{\mathcal{Q}}%
\global\long\def\endoofvlambda{\normxilambda\otimes\check{\xi}_{\lambda}^{\mathrm{norm}}}%
\global\long\def\cntwok{\left(\C^{n}\right)^{\otimes2k}}%
\global\long\def\vsigmas{v_{\sigma(1)}\otimes\dots\otimes v_{\sigma(k)}\otimes v_{\sigma'(1)}\otimes\dots\otimes v_{\sigma'(k)}}%
\global\long\def\D{\mathcal{D}}%
\global\long\def\tablambda{\mathrm{Tab}_{\mathrm{\normxilambda}}\left(\lambda^{+}(2k)\right)}%
\global\long\def\part{\mathrm{Part}\left([2k]\right)}%
\global\long\def\ei{e_{i_{i}}\otimes\dots\otimes e_{i_{k}}}%
\global\long\def\csk{\C\left[S_{k}\right]}%
\global\long\def\mixedtensor{\cnk\otimes\left(\left(\C^{n}\right)^{\vee}\right)^{\otimes l}}%
\global\long\def\shortmixedtensor{C_{k,l}^{n}}%
\global\long\def\f{F_{r}}%
\global\long\def\b{\mathcal{B}_{r}^{\times}}%
\global\long\def\u{\mathcal{U}_{\lambda,n}}%
\global\long\def\refinedpart{\overset{\star}{\mathrm{Part}}\left(\left[|w|_{f}k\right]\right)}%
\global\long\def\skpart{\mathrm{Part}_{\leq S_{k}}\left([2k]\right)}%
\global\long\def\gsigma{G\left(\sigma_{x},\tau_{x},\pi_{i}\right)}%
\global\long\def\hatgsigma{\hat{G}\left(\sigma_{x},\tau_{x},\pi_{i}\right)}%
\global\long\def\gamsig{\Gamma\left(\sigma_{x},\tau_{x},\pi_{i}\right)}%
\global\long\def\tgamsig{\tilde{\Gamma}\left(\sigma_{x},\tau_{x},\pi_{i}\right)}%
\global\long\def\Astack{\mathcal{A}_{\hat{\Lambda}'}}%
\global\long\def\B{\mathcal{B}_{\hat{\Lambda}'}}%
\global\long\def\integral{\mathcal{I}_{n}(w,\lambda,g)}%
\global\long\def\samgamsig{\Gamma\left(\sigma_{x},\sigma_{x},\pi_{i}\right)}%
\global\long\def\link{\mathrm{Link}(v)}%
\global\long\def\alphabet{\{a_{1},b_{1},\dots,a_{g},b_{g}\}}%
\global\long\def\match{\mathrm{Match}\left(w,k\right)}%

\title{Word maps and surface relations in symmetric groups}
\author{Ewan Cassidy}
\maketitle
\begin{abstract}
We study the expected number of fixed points of a random permutation
obtained via a word map, with surface group constraints imposed. For
stable irreducible characters $\chi$ of the symmetric group $S_{n}$
and with $R_{g}=[a_{1},b_{1}]\dots[a_{g},b_{g}]$ and $w\in F_{2g}$,
we compute $\mathbb{E}_{S_{n}^{2g}}\left[\chi\left(R_{g}(h)\right)\#\mathrm{fix}\left(w(h)\right)\right].$
We show that, if $w$ is a shortest representative for the conjugacy
class of $\gamma\in\Gamma_{g}=\left\langle a_{1},b_{1},\dots,a_{g},b_{g}:R_{g}\right\rangle $,
then this expectation is $O\left(1/\dim\chi\right).$ As an application,
we recover a boundedness statement of Magee--Puder on the large $n$
limit of the expected number of fixed points of $\phi_{n}(\gamma),$
where $\gamma\in\Gamma_{g}$ is fixed and $\phi_{n}\in\hom\left(\Gamma_{g},S_{n}\right)$
is chosen uniformly at random. 
\end{abstract}

\section{\label{sec:Introduction}Introduction}

Given a word $w$ in the free group $F_{r}=\left\langle a_{1},\dots,a_{r}\right\rangle $,
one can define a word map 
\[
w:G^{r}\to G
\]
for any compact group $G$ via substitutions of the $a_{i}$ by group
elements $g_{i}.$ This paper is focused on the case $G=S_{n}.$ Now,
fix an integer $g\ge2$ . We are interested in multi--trace statistics
for a pair of words, the first word being some fixed $w$ in the free
group
\[
F_{2g}=\left\langle a_{1},b_{1},\dots,a_{g},b_{g}\right\rangle 
\]
and the second word being 
\[
R_{g}\overset{\mathrm{def}}{=}\left[a_{1},b_{1}\right]\dots\left[a_{g},b_{g}\right].
\]
More specifically, writing $h=(h_{a_{1}},h_{b_{1}},\dots,h_{a_{g}},h_{b_{g}})\in S_{n}^{2g}$,
we are interested in integrals of the following form 
\begin{equation}
\int_{h}\chi\left(R_{g}\left(h\right)\right)\mathrm{Tr}\left(w(h)\right)dh,\label{eq: integral}
\end{equation}
where $\chi$ is some irreducible character of the symmetric group
and $\mathrm{Tr}(w(h))$ is the trace of $w(h)$ in the defining representation
of $S_{n}$ (i.e. the trace as a $0$--$1$ permutation matrix).
One can equally interpret $\mathrm{Tr}(w(h))$ as the number of fixed
points of the permutation $w(h).$ These integrals are done with respect
to the uniform probability measure on $S_{n}^{2g}.$ Motivated by
their connection to fixed point statistics for the image $\phi_{n}(\gamma)$
for $\gamma$ an element of 
\[
\Gamma_{g}\overset{\mathrm{def}}{=}\left\langle a_{1},b_{1},\dots,a_{g},b_{g}:\left[a_{1},b_{1}\right]\dots\left[a_{g},b_{g}\right]\right\rangle ,
\]
the fundamental group of a closed orientable genus $g$ surface $\Sigma_{g}$
and $\phi_{n}\in\hom\left(\Gamma_{g},S_{n}\right)$ chosen uniformly
at random, we are particularly interested in giving asymptotic bounds
on (\ref{eq: integral}) when $\chi$ is a \emph{stable irreducible
character}. These are irreducible representations of $S_{n}$ indexed
by Young diagrams with a fixed number of boxes outside the first row.
That is, for some fixed some positive integer $k$ and Young diagram
$\lambda\vdash k$ , then $\chi$ is indexed by a Young diagram $\lambda^{+}(n)$
with $n$ boxes (for $n\ge2k)$ defined by 
\[
\lambda^{+}(n)\overset{\mathrm{def}}{=}\left(n-k,\lambda\right)\vdash n,
\]
i.e. by adding an additional row of $n-k$ boxes to the top of $\lambda$.
These representations are of interest because they characterize the
lowest dimension representations of $S_{n}$ of dimension $\sim n^{k}$,
see \S\ref{subsec:Representation-theory-of}. Although these integrals
can be defined for arbitrary words $w$, to obtain good asymptotic
bounds using the methods of this paper we depend on the (lack of)
interaction between $w$ and the surface relation $R_{g}.$ In particular,
we require that $w$ cannot be shortened via cyclic reduction or by
using the relation $R_{g}=e.$ Our first result is the following theorem. 
\begin{thm}
\label{thm: bound on integral main}Fix $\gamma\in\Gamma_{g}$ with
$\gamma\neq e.$ For any word $w$ that is a shortest representative
of the conjugacy class of $\gamma$ in $\Gamma_{g},$ for any fixed
$k\in\mathbb{Z}_{\ge0}$ and for any $\lambda\vdash k$, we have 
\[
\mathcal{I}_{n}(w,\lambda,g)\overset{\mathrm{def}}{=}\int_{S_{n}^{2g}}\chi^{\lambda^{+}(n)}\left(R_{g}(h)\right)\mathrm{Tr}\left(w(h)\right)dh=O_{k,\ell(w),g}\left(\frac{1}{d_{\lambda^{+}(n)}}\right)=O_{k,\ell(w),g}\left(\frac{1}{n^{k}}\right),
\]
where $d_{\lambda^{+}(n)}=\dim\chi^{\lambda^{+}(n)}.$
\end{thm}

Without further conditions on the choice of representing word $w$
for the conjugacy class of $\gamma,$ it is not possible to obtain
a better bound than that of Theorem \ref{thm: bound on integral main}
using the approach in this paper, as can be seen in Example \ref{exa: -k example}.
Nevertheless, for a large family of $\gamma\in\Gamma_{g},$ we can
give the following improvement for $\lambda\vdash k$ with $k\ge1$,
with the proof given in \S\ref{subsec:Proof-of-Lemma DEGENERATE}.
We note that, for $k=0,$ $\lambda^{+}(n)$ indexes the trivial representation
of $S_{n},$ and thus $\mathcal{I}(w,\mathrm{triv},g)$ is the expected
number of fixed points of a $w$--random permutation, where much
stronger bounds are already known from \cite{Nica,LinialPuder,PuderParzanchewski2012}. 
\begin{thm}
\label{thm: -k-1 bound}Fix $\gamma\in\Gamma_{g}$ with $\gamma\neq e.$
Suppose that, up to taking cyclic permutations, the conjugacy class
of $\gamma$ in $\Gamma_{g}$ has a unique shortest representative
word $w$. Then, for any fixed $k\in\mathbb{Z}_{\ge1}$ and for any
$\lambda\vdash k$, we have 
\[
\integral=O_{k,\ell(w),g}\left(\frac{1}{n^{k+1}}\right).
\]
\end{thm}

In principle, the combinatorial integration method and the subsequent
reduction to a `core graph type' expansion described in \S\ref{sec:Combinatorial-integration}
can be used to describe the integrals 
\begin{equation}
\int_{S_{n}^{2g}}\chi^{\lambda^{+}(n)}\left(R(h)\right)\mathrm{Tr}\left(w(h)\right)dh\label{eq: integral for general relators}
\end{equation}
for any word $w$ and arbitrary relator $R$. As is described in \S\ref{subsec:Bounding},
to obtain a useful asymptotic bound as in Theorem \ref{thm: bound on integral main},
we would require a combinatorial characterisation of when $w$ is
a shortest representing word in its conjugacy class in $\Gamma=\left\langle a_{1},\dots,a_{r}:R\right\rangle .$
In the case of $\Gamma=\Gamma_{g}$ as above, this ingredient is readily
available and goes back to the work of Dehn \cite{Dehn} and Birman--Series
\cite{BirmanSeries}. It is plausible that obtaining a similar characterisation
for more general relators $R$ and using the methods of this paper
could yield similarly strong estimates on (\ref{eq: integral for general relators}).

The following section details an important application of Theorem
\ref{thm: bound on integral main} towards computing asymptotic statistics
for random covers of closed surfaces. In particular, it makes clear
why the bound obtained in Theorem \ref{thm: bound on integral main}
on $\integral$ must `beat the dimension' of the corresponding irreducible
representation. 

\subsection{\label{subsec:Applications:-Expected-number}Application: Expected
number of fixed points}

In \cite{MageePuder2023}, Magee and Puder evaluate the following
expectations 
\begin{equation}
\mathop{\mathbb{E}_{\phi_{n}}\left[\#\mathrm{fix}\left(\phi_{n}(\gamma)\right)\right],}\label{eq: expected trce surface}
\end{equation}
where $g\ge2$ is fixed, $\Gamma_{g}$ is the fundamental group of
a closed orientable genus $g$ surface $\Sigma_{g}$, $\gamma\in\Gamma_{g}$
(with $\gamma\neq e$), the expectation is taken over all $\phi_{n}\in\hom\left(\Gamma_{g},S_{n}\right)$
chosen uniformly at random and $\#\mathrm{fix}$ is the number of
fixed points. We note that the representation space 
\[
\mathbb{X}_{g,n}\overset{\mathrm{def}}{=}\hom\left(\Gamma_{g},S_{n}\right)
\]
is finite (see \cite[Proposition 3.2]{LiebeckShalev} for example)
and so taking the uniform probability measure is natural. This representation
space can be equally viewed as the space of all $n$--sheeted covers
of the surface $\Sigma_{g}$, as is discussed in \cite{MageePuder2023}.
The expectation (\ref{eq: expected trce surface}) is then the expected
number of times a closed loop $\gamma$ in $\Sigma_{g}$ lifts to
a closed loop in a random $n$--sheeted cover. Magee and Puder proved
that, for large $n$, this expectation can be approximated to any
order by a rational function in $n,$ i.e. that it has an asymptotic
expansion and determined the large $n$ limit. 
\begin{thm}[{\cite[Theorems 1.2 and 1.3]{MageePuder2023}}]
\label{thm:Magee-Puder main theorem} Fix $\gamma\in\Gamma_{g}$.
There is an infinite sequence of numbers 
\[
a_{-1},a_{0},a_{1},a_{2},\dots
\]
depending on $\gamma,$ such that, for any $m\in\mathbb{Z}_{>0}$,
as $n\to\infty,$ 
\[
\mathbb{E}_{\phi_{n}}\left[\#\mathrm{fix}\left(\phi_{n}(\gamma)\right)\right]=a_{-1}n+a_{0}+a_{1}n^{-1}+a_{2}n^{-2}+\dots+a_{m-1}n^{-(m-1)}+O\left(n^{-m}\right).
\]
 Moreover, $a_{-1}\neq0$ if and only if $\gamma=e,$ and if $b\in\mathbb{Z}_{>0}$
is maximal such that $\gamma=\delta^{b}$ for some $\delta\in\Gamma_{g},$
then 
\[
a_{0}=d(b),
\]
where $d(b)$ is the number of divisors of $b$. 
\end{thm}

The primary application of Theorem \ref{thm: bound on integral main}
is towards giving a more accessible proof of the boundedness statement,
namely the statement that $a_{-1}=0$, whenever $\gamma$ is not the
identity, from Theorem \ref{thm:Magee-Puder main theorem}. In the
setting of Theorem \ref{thm: -k-1 bound}, we can recover Theorem
\ref{thm:Magee-Puder main theorem} in full. We now give a rough overview
of the initial approach of \cite{MageePuder2023}, further detailed
in \cite{MageePuderVanHandel}, to explain how this can be done. The
goal is to compute 
\[
\mathbb{E}_{\phi_{n}}\left[\#\mathrm{fix}\left(\phi_{n}(\gamma)\right)\right]=\frac{1}{|\mathbb{X}_{g,n}|}\sum_{\phi_{n}\in\mathbb{X}_{g,n}}\#\mathrm{fix}\left(\phi_{n}(\gamma)\right).
\]
By choosing a representing word $w$ for the conjugacy class of $\gamma$
in $\Gamma_{g}$, we can rewrite this as 
\[
\frac{1}{|\mathbb{X}_{g,n}|}\sum_{\psi_{n}\in\hom\left(F_{2g},S_{n}\right)}\#\mathrm{fix}\left(\psi_{n}(w)\right)\mathds{1}\left[\psi_{n}\left(R_{g}\right)=e\right]
\]
and then Fourier expand the indicator function to obtain 
\[
\frac{1}{|\mathbb{X}_{g,n}|}\frac{1}{n!}\sum_{\psi_{n}\in\hom\left(F_{2g},S_{n}\right)}\#\mathrm{fix}\left(\psi_{n}(w)\right)\sum_{\lambda\vdash n}d_{\lambda}\chi^{\lambda}\left(\psi_{n}\left(R_{g}\right)\right).
\]
This is nothing but 
\begin{equation}
\frac{|S_{n}^{2g}|}{n!|\mathbb{X}_{g,n}|}\sum_{\lambda\vdash n}d_{\lambda}\int_{S_{n}^{2g}}\chi^{\lambda}\left(R_{g}\left(h\right)\right)\mathrm{Tr}\left(w(h)\right)dh.\label{eq: expressing expected trace as FULL sum of integrals}
\end{equation}
From e.g. \cite{LiebeckShalev,Hurwitz}, we have that $|\mathbb{X}_{g,n}|=(n!)^{2g-1}\zeta^{S_{n}}\left(2g-2\right),$
where for a finite group $G,$ $\zeta^{G}(s)\overset{\mathrm{def}}{=}\sum_{\chi\in\mathrm{Irr}(G)}\frac{1}{\left(\dim\chi\right)^{s}}$
is the Witten--Zeta function. Thus, from (\ref{eq: expressing expected trace as FULL sum of integrals})
we obtain 
\begin{equation}
\mathbb{E}_{\phi_{n}}\left[\#\mathrm{fix}\left(\phi_{n}(\gamma)\right)\right]=\frac{1}{\zeta^{S_{n}}\left(2g-2\right)}\sum_{\lambda\vdash n}d_{\lambda}\int_{h}\chi^{\lambda}\left(R_{g}\left(h\right)\right)\mathrm{Tr}\left(w(h)\right)dh.\label{eq: non truncated sum over irreps}
\end{equation}

The next step is to split the sum in (\ref{eq: non truncated sum over irreps})
to a sum over irreducible representations with dimension at most $n^{k}$
for some fixed $k$ and a sum over the higher--dimensional representations.
These two sums can then be dealt with separately, and the latter sum's
contribution to (\ref{eq: non truncated sum over irreps}) can be
bounded using a much simplified version of the approach in \cite{MageePuder2023}.
This can be done as in \cite[Proposition 4.4]{MageePuderVanHandel},
where they show that, if $\ell(w)\le q$ with $\ell(w)$ the word
length of $w$, then (\ref{eq: non truncated sum over irreps}) is
equal to 
\begin{equation}
\frac{2}{\zeta^{S_{n}}\left(2g-2\right)}\sum_{|\lambda|\le4q}d_{\lambda^{+}(n)}\integral+O\left(\left(\frac{(Cq)^{C}}{n}\right)^{q}\right).\label{eq: truncated expression}
\end{equation}
We note that this effective error bound is far stronger than what
is needed for our purposes, and holds uniformly for all $n\ge Cq^{C}.$
For our purposes, $q$ will always be fixed, and we can absorb the
factor $(Cq)^{Cq}$ in to the implied constant, and instead write
(\ref{eq: truncated expression}) as 
\begin{equation}
\frac{2}{\zeta^{S_{n}}\left(2g-2\right)}\sum_{|\lambda|\le4q}d_{\lambda^{+}(n)}\integral+O_{q}\left(n^{-q}\right).\label{eq: bound with non effective error term and truncated expression}
\end{equation}
The application of our Theorem \ref{thm: bound on integral main}
is to bound the contribution from the irreducible representations
corresponding to $\chi^{\lambda^{+}(n)}$, denoted $V^{\lambda^{+}(n)}$,
with $|\lambda|\le4q$ that appear in the former sum. Given that $d_{\lambda^{+}(n)}\sim n^{|\lambda|},$
for each fixed $\lambda\vdash k$ this shows why, in Theorem \ref{thm: bound on integral main},
we must obtain $O\left(\frac{1}{d_{\lambda^{+}(n)}}\right)$ or better.
We again stress that $w$ and therefore $q$ is fixed beforehand and
thus this is always a finite sum and it suffices to bound the contribution
from each individual representation. In \cite[Section 4.3]{MageePuderVanHandel},
the authors analyse $\integral$ for each $\lambda$ with $|\lambda|\le4q,$
and using a similar approach to our combinatorial integration approach
(see \S\ref{sec:Combinatorial-integration}), they show that the
contribution to (\ref{eq: bound with non effective error term and truncated expression})
agrees with a rational function of $n$ whenever $n$ is large enough.
Our main result shows that with this approach, one can actually obtain
the large--$n$ boundedness statement contained in Theorem \ref{thm:Magee-Puder main theorem}.
\begin{cor}
In the same setting as in Theorem \ref{thm:Magee-Puder main theorem},
there is an infinite sequence of numbers 
\[
a_{-1},a_{0},a_{1},a_{2},\dots
\]
depending on $\gamma,$ such that, for any $m\in\mathbb{Z}_{>0}$,
as $n\to\infty,$ 
\[
\mathbb{E}_{\phi_{n}}\left[\#\mathrm{fix}\left(\phi_{n}(\gamma)\right)\right]=a_{-1}n+a_{0}+a_{1}n^{-1}+a_{2}n^{-2}+\dots+a_{m-1}n^{-(m-1)}+O\left(n^{-m}\right)
\]
and, if $\gamma\neq e,$ then $a_{-1}=0.$

If, in addition, $\gamma\in\Gamma_{g}$ with $\gamma=\delta^{b}$
with $b$ maximal is such that, up to taking cyclic permutations,
the conjugacy class of $\gamma$ in $\Gamma_{g}$ has a unique shortest
representative word $w$, then $a_{0}=d(b).$
\end{cor}

\begin{proof}
Since, for each $\lambda\vdash k,$ we have $d_{\lambda^{+}(n)}=O\left(n^{k}\right),$
Theorem \ref{thm: bound on integral main} combined with the fact
$\frac{2}{\zeta^{S_{n}}\left(2g-2\right)}\to1$ (from e.g. \cite[Proposition 1.1]{LiebeckShalev})
gives 
\[
\frac{2}{\zeta^{S_{n}}\left(2g-2\right)}\sum_{|\lambda|\le4q}d_{\lambda^{+}(n)}\integral=O_{q,g}(1)
\]
 as $n\to\infty$ immediately. 

In the additional setting, we note that such a shortest representative
word $w$ must also be a proper power in $F_{2g}$, that is we must
have $w=u^{b},$ where $u$ is the unique (up to cyclic permutation)
shortest representative for the conjugacy class of $\delta.$ Theorem
\ref{thm: -k-1 bound} then says that (\ref{eq: bound with non effective error term and truncated expression})
is equal to
\[
\frac{2}{\zeta^{S_{n}}\left(2g-2\right)}\mathcal{I}\left(u^{b},\emptyset,g\right)+O_{q,g}\left(\frac{1}{n}\right).
\]
Then, $\mathcal{I}\left(u^{b},\emptyset,g\right)$ is just the expected
number of fixed points of $u^{b}(h)$, which is exactly $d(b)+O_{q}\left(\frac{1}{n}\right)$
by e.g. \cite{LinialPuder}. 
\end{proof}

\subsection{\label{subsec:Other-related-work}Related work and further questions}

\subsubsection*{Unitary groups}

The most closely related results to the current paper are those of
Magee in \cite{Magee2021}, which builds on \cite{Magee22}. In \cite[Theorem 3.1]{Magee2021},
Magee proves the bound 
\[
\mathcal{J}_{n}(w,\mu,\nu)\overset{\mathrm{def}}{=}\int_{U(n)^{2g}}s_{\mu,\nu}\left(R_{g}(h)\right)\mathrm{Tr}\left(w(h)\right)dh\ll n^{-k-\ell},
\]
where $\mu\vdash k,\nu\vdash\ell$ and $s_{\mu,\nu}$ is the Schur
polynomial corresponding to the character of the irreducible representation
of $U(n)$ with signature $\left(\mu_{1},\dots,\mu_{\ell(\mu)},\underbrace{0,\dots,0}_{n-k-\ell},-\nu_{\ell(\nu)},\dots,-\nu_{1}\right).$
This is a direct unitary analogue of our Theorem \ref{thm: bound on integral main}.
Our approach is inspired by Magee's combinatorial integration approach.
In fact, Magee obtains the above bound by reducing the problem to
bounding the Euler characteristic of a family of surfaces. The dual
graphs of these surfaces are special cases of the graphs we construct
in \S\ref{subsec:A-core-graph}.

\subsubsection*{Improved bound}

One could try to improve the bound in Theorem \ref{thm: bound on integral main}
in order to obtain the $a_{0}=d(b)$ term in Theorem \ref{thm:Magee-Puder main theorem}
for all $\gamma\in\Gamma_{g}$. Interestingly, in \S\ref{subsec:Bounding},
we show that assuming that $w$ is a shortest representative for the
conjugacy class of $\gamma$ is not sufficient to improve on the bound
in Theorem \ref{thm: bound on integral main}. However, it seems plausible
that a more careful choice of shortest representative $w$ could yield
better bounds on $\integral$ for $k\ge1$, from which one could extract
the $a_{0}=d(b)$ term from $\mathcal{I}_{n}(w,\emptyset,g),$ just
as we do in the case of the conjugacy class of $\gamma$ having a
unique shortest representative word $w$. This will be the subject
of future investigation. 

\subsubsection*{Sharper estimates}

Consider the same question but with $\Gamma_{g}$ replaced by a free
group $F_{r}.$ For a fixed word $w\in F_{r},$ 
\begin{equation}
\mathbb{E}_{\phi_{n}\in\hom\left(F_{r},S_{n}\right)}\left[\#\mathrm{fix}\left(\phi_{n}(w)\right)\right]\label{eq: expected trace for free}
\end{equation}
is the expected number of times the loop $w$ in $\mathcal{B}_{r}=\overset{\mathrm{def}}{=}\bigvee_{i=1}^{r}S_{1}$
lifts to a closed loop in a uniformly random $n$--sheeted cover.
In this setting, initial bounds of the form $d(b)+O\left(\frac{1}{n}\right)$
were obtained by \cite{Nica,LinialPuder}. Much sharper results were
later obtained by Puder and Parzanchevski \cite{PuderParzanchewski2012},
who showed that the first non--zero coefficient in the asymptotic
expansion for (\ref{eq: expected trace for free}) is determined by
an algebraic property of the fixed word $w$ called the \emph{primitvity
rank}. It is an interesting question to try and determine the first
non--zero term in the asymptotic expansion of (\ref{eq: expected trce surface})
for surface groups.\footnote{In both settings, the expected number of fixed points is $1+\mathbb{E}\left[\chi^{\mathrm{std}}\left(\phi_{n}(w)\right)\right],$
so we really mean the first non--zero term in the asymptotic expansion
of the latter term.}

\subsubsection*{Alternative characters}

One can also consider replacing $\#\mathrm{fix}$ with any character
$\chi$ of the symmetric group. Again in the free group setting, asymptotic
bounds have been obtained by Hanany and Puder in \cite{HananyPuder}
and also by the author in \cite{Cassidy2} for all stable representations.
For surface groups, for any character $\chi$ of $S_{n}$ and $\gamma\in\Gamma_{g},$
the functions 
\[
\phi_{n}\mapsto\chi\left(\phi_{n}\left(\gamma\right)\right)
\]
are called Wilson loops in the physics literature. Theorem \ref{thm:Magee-Puder main theorem}
is focused on the character $\#\mathrm{fix}.$ It may be interesting
to determine whether the approach to Theorem \ref{thm:Magee-Puder main theorem}
via the methods of this paper can be applied for other characters
$\chi.$

\subsubsection*{Other discrete groups}

The current paper reveals the primary obstacles to generalizing the
combinatorial integration method towards obtaining similar statements
to that of Theorem \ref{thm:Magee-Puder main theorem} for more general
one--relator groups $\Gamma$. As mentioned previously, the primary
obstacle to obtaining useful bounds on integrals of the form $\integral.$
with $R_{g}$ replaced by arbitrary relator $R,$ is a combinatorial
characterisation of a word $w$ being a shortest representative of
the conjugacy class of $\gamma\in\Gamma.$ Again, in the present paper,
this step depends on $\Gamma_{g}$ being the fundamental group of
a closed orientable surface. 

For the application towards estimating $\mathbb{E}_{\phi_{n}}\left[\#\mathrm{fix}\left(\phi_{n}(\gamma)\right)\right]$
using the approach detailed in the previous section, there are two
further obstacles. 
\begin{itemize}
\item \textbf{Truncation: }The integrals we compute in this paper for surface
groups apply only when the integrand consists of $\#\mathrm{fix}(w(h))$
multiplied by some \emph{stable} irreducible character of the image
of the relator. A key step in using these integrals in the application
towards computing the expected number of fixed points is bounding
the contribution from the higher dimensional representations. For
surface groups, this can be done in \cite{MageePuderVanHandel} using
a simplified version of the approach in \cite{MageePuder2023} and,
at present, appears to rely on the group relator having a very specific
structure. 
\item \textbf{Normalization: }Another key ingredient is being able to evaluate
the large $n$ limit of $|\hom\left(\Gamma,S_{n}\right)|.$ For surface
groups (and more generally Fuchsian groups) this was already known,
but is not known in full generality. 
\end{itemize}
We highlight here some forthcoming work of Klukowski \cite{Klukowski}.
Independently and concurrently with the present work, Klukowski uses
a different approach and generalizes the statement of Theorem \ref{thm:Magee-Puder main theorem}
to arbitrary Fuchsian groups.

\section{\label{sec:Background}Background}

\subsection{\label{subsec:Representation-theory-of}Representation theory of
$S_{n}$ }

The equivalence classes of finite dimensional complex irreducible
representations of $S_{n}$ are indexed by Young diagrams $\lambda$
with $n$ boxes. We write $\lambda=(\lambda_{1},\dots,\lambda_{\ell(\lambda)})\vdash n$
for the Young diagram $\lambda$ with $\ell(\lambda)$ non--empty
rows of boxes, with row $i$ containing $\lambda_{i}$ boxes and $\lambda_{1}\ge\lambda_{2}\ge\dots\lambda_{\ell(\lambda)}>0$
and $\sum_{i}\lambda_{i}=n.$ For fixed $k$ and each Young diagram
$\lambda\vdash k,$ we define a family of irreducible representations
of $\{S_{n}\}_{n\ge2k}$ by 
\[
\{\lambda^{+}(n)\}_{n\ge2k},
\]
whereby $\lambda^{+}(n)$ is a Young diagram with $n$ boxes, obtained
by adding a row of $n-k$ boxes to the top of $\lambda.$ That is,
if $\lambda=(\lambda_{1},\dots,\lambda_{\ell(\lambda)})\vdash k,$
then\footnote{Strictly speaking, this family can be defined as long as $n\ge k+\lambda_{1}.$
For our purposes, this will not be important. }
\[
\lambda^{+}(n)\overset{\mathrm{def}}{=}\left(n-k,\lambda_{1},\dots,\lambda_{\ell(\lambda)}\right)\vdash n.
\]
Such representations are called stable representations and they have
dimension $d_{\lambda^{+}(n)}\sim n^{k}$. They can be constructed
as explicit subspaces (with a known multiplicity) of $\cnk$, where
$\C^{n}$ is the defining permutation representation of $S_{n}$,
via Schur--Weyl--Jones duality. It is this construction that will
allow us to determine bounds on $\integral$. An alternative approach
to working with these representations is discussed in \cite[Appendix B]{HananyPuder},
where they can be encoded by considering cycle counts of permutations.
We note that these representations and their conjugates (obtained
by transposing the rows and columns of the Young diagram) characterise
all irreducible representations of $S_{n}$ of dimension $\le Cn^{k}$,
see e.g. \cite[Lemma 2.8]{Etingof}.

\subsection{\label{subsec:A-projection-formula}A projection formula}

Here we describe a projection formula obtained by the author in \cite{Cassidy2023}.
For each positive integer $N$, denote by $\mathrm{Part}\left([N]\right)$
the set of set partitions of $[N].$ On $\cnk$ there is a well known
Schur--Weyl--Jones duality due to Jones \cite{Jones}, by which
it is known that the diagonal action of the symmetric group $S_{n}$
and the action of the partition algebra $\P$ (an algebra that consists
of linear combinations of elements of $\part$ with a natural multiplication
and action on $\cnk$ defined) generate full mutual centralizers of
one another. There is a refinement of this due to Sam and Snowden
\cite{SamSnowden}. Define the following subspace 
\[
A_{k}(n)\overset{\mathrm{def}}{=}\left\langle e_{i_{1}}\otimes\dots\otimes e_{i_{k}}:\ i_{1},\dots,i_{k}\ \mathrm{all\ distinct}\right\rangle \cap\bigcap_{j=1}^{k}\ker(T_{j}),
\]
where $T_{j}:\cnk\to\left(C^{n}\right)^{\otimes k-1}$ is a contraction
map, deleting the $j^{\mathrm{th}}$ coordinate. By \cite[Section 6.1.3]{SamSnowden}
(or see e.g. \cite{Cassidy2023}), on $\A,$ the action of $\P$ descends
to an action of an embedded copy of $\csk$ via 

\[
S_{k}\overset{\iota}{\hookrightarrow}\part,
\]
viewing each permutation as a matching of $\{1,\dots,k\}$ with $\{k+1,\dots,2k\}$
in to $k$ pairs. One can view these as diagrams consisting of $2$
rows of $k$ vertices with $k$ edges pairing vertices in opposing
rows. Moreover, we get a decomposition of $\A$ as a $\C\left[S_{n}\times S_{k}\right]$
representation 
\[
\A\cong\bigoplus_{\lambda\vdash k}V^{\lambda^{+}(n)}\otimes V^{\lambda}.
\]
Hence, each irreducible subspace in the above decomposition is isomorphic
to $d_{\lambda}$ copies of $V^{\lambda^{+}(n)}$ as a representation
of $S_{n}.$ 

Given a set partition $\pi\in\part,$ define an endomorphism $P_{\pi}\in\mathrm{End}\left(\cnk\right)$
by 
\[
\left\langle P_{\pi}\left(e_{i_{1}}\otimes\dots\otimes e_{i_{k}}\right),e_{i_{k+1}}\otimes\dots\otimes e_{i_{2k}}\right\rangle \overset{\mathrm{def}}{=}\begin{cases}
1 & \mathrm{if}\ i_{j}=i_{\ell}\ \mathrm{iff}\ j\sim\ell\ \mathrm{in}\ \pi\\
0 & \mathrm{otherwise}.
\end{cases}
\]
We note that, if $\pi=\iota(\tau)$ for some $\tau\in S_{k},$ then
$P_{\pi}$ permutes the tensor coordinates according to $\tau.$ We
write $\pi_{1}\le\pi_{2}$ if every block of $\pi_{1}$ can be obtained
from $\pi_{2}$ by splitting the blocks of $\pi_{2}$ in to smaller
blocks. For $\pi\le\iota(\tau)$, this means that the diagram for
$\pi$ can be obtained from that of $\iota(\tau)$ by removing edges.
We will write 
\[
\sum_{\pi\in\part}^{\le S_{k}}
\]
to indicate that we are summing over all set partitions $\pi\in\part$
that are $\le\iota(\tau)$ for any $\tau\in S_{k}.$ These are all
set partitions for which, in the corresponding diagram, each block
contains at most one vertex from each row. We write $|\pi|$ for the
number of blocks of $\pi.$ Below is the formula obtained by the author
in \cite{Cassidy2023}. 
\begin{prop}
\label{prop: projection formula}For any $k\in\mathbb{Z}_{\ge0}$
and for any $\lambda\vdash k$ and $n\ge2k$, 
\[
\mathcal{Q}_{\lambda,n}\overset{\mathrm{def}}{=}\sum_{\pi\in\part}^{\le S_{k}}c(n,k,\lambda,\pi)P_{\pi},
\]
with 
\[
c(n,k,\lambda,\pi)\overset{\mathrm{def}}{=}\frac{d_{\lambda^{+}(n)}(-1)^{|\pi|+k}}{(n)_{|\pi|}}\sum_{\tau\in S_{k},\iota(\tau)\ge\pi}\chi^{\lambda}(\tau),
\]
is the orthogonal projection from $\cnk$ to $d_{\lambda}$ copies
of $V^{\lambda^{+}(n)}.$
\end{prop}

Given $\pi\in\part$ with $\pi\le\iota(\tau),$ we write 
\[
\mathrm{del}(\pi)\overset{\mathrm{def}}{=}|\pi|-k.
\]
The notation reflects that this quantity is exactly the number of
edges deleted from the diagram for $\iota(\tau)$ to obtain the diagram
for $\pi.$ In particular, $(n)_{|\pi|}\sim n^{|\pi|}=n^{\mathrm{del}(\pi)+k},$
and given that $d_{\lambda^{+}(n)}\sim n^{k}$ for $\lambda\vdash k,$
we have that 
\begin{equation}
c(n,k,\lambda,\pi)=O\left(n^{-\mathrm{del}(\pi)}\right).\label{eq: coeff in projection is O(-del)}
\end{equation}

\subsection{\label{subsec:The-Weingarten-calculus}The Weingarten calculus for
$S_{n}$}

The Weingarten calculus is a powerful tool for integrating products
of matrix coefficients over compact groups, based on Schur--Weyl
duality. Initially developed by Weingarten \cite{Weingarten} and
Xu \cite{Xu}, it was then put on a more rigorous footing by Collins
\cite{Collins2003} and Collins and Sniady \cite{CollinsSniady2006}.
In the case of integrating over $S_{n}$, the method reduces to elementary
combinatorics, see e.g. \cite[Section 3]{CollinsWeingartenShort}.
We give here the formulation of the Weingarten calculus for $S_{n}$
to be used in this paper. Denote the diagonal action of $S_{n}$ on
$\cnk$ by the representation 
\[
\rho:S_{n}\to\mathrm{End}\left(\cnk\right).
\]
For simplicity, we will write $g_{ij}$ for the matrix coefficient
$\rho(g)_{ij}.$ We define, for any multi--index $I=(i_{1},\dots,i_{k})$
and for any set partition $\pi\in\mathrm{Part}([k]),$ a map
\[
\pi\left(I\right)=\begin{cases}
1 & \mathrm{if},\ j\sim l\ \mathrm{in}\ \pi\implies i_{j}=i_{l}\\
0 & \mathrm{otherwise.}
\end{cases}
\]

\begin{thm}
\label{thm: Weingarten theorem}For any multi--indices $I=(i_{1},\dots,i_{k})$
and $J=(j_{1},\dots,j_{k})\in[n]^{k},$ we have 
\[
\int_{S_{n}}\prod_{\ell=1}^{k}g_{i_{\ell}j_{\ell}}=\sum_{\sigma,\tau\in\mathrm{Part}([k])}\sigma\left(I\right)\tau\left(J\right)\mathrm{Wg}_{n,k}(\sigma,\tau),
\]
where 
\[
\mathrm{Wg}_{n,k}(\sigma,\tau)\overset{\mathrm{def}}{=}\sum_{\pi\le\sigma\wedge\tau}\mu(\pi,\sigma)\mu(\pi,\tau)\frac{1}{(n)_{|\pi|}}
\]
is the Weingarten function for $S_{n}.$ 
\end{thm}

In the above Theorem, $\mu$ is the M{\"o}bius function for $\mathrm{Part}([k])$,
which can be defined as in \cite{Rota}. We will use the bound $\mathrm{Wg}_{n,k}(\sigma,\tau)=O\left(\frac{1}{n^{|\sigma\wedge\tau|}}\right)$
when we take expectations in \S\ref{sec:Combinatorial-integration}.
Given two multi--indices $I=(i_{1},\dots i_{k_{I}})$ and $J=(j_{1},\dots,j_{k_{J}}),$
we will write $I\sqcup J$ for the multi--index given by $(i_{1},\dots i_{k_{I}},j_{1},\dots,j_{k_{J}}).$
We will also use the notation $e_{I}=e_{i_{1}}\otimes\dots\otimes e_{i_{k}}.$

\subsection{\label{subsec:Free-groups-and}Free groups and surface groups}

For each fixed integer $g\ge2,$ we will write $F_{2g}$ for the free
group generated by $\{a_{1},b_{1},\dots,a_{g},b_{g}\},$ that is 
\[
F_{2g}\overset{\mathrm{def}}{=}\left\langle a_{1},b_{1},\dots,a_{g},b_{g}\right\rangle .
\]
We will also write 
\[
R_{g}\overset{\mathrm{def}}{=}\left[a_{1},b_{1}\right]\dots\left[a_{g},b_{g}\right],
\]
with the convention $[a,b]=aba^{-1}b^{-1}.$ For ease, sometimes we
will fix $g=2$ and write $F_{4}=\left\langle a,b,c,d\right\rangle $
and $R_{2}=[a,b][c,d].$ We denote by $\Gamma_{g}\overset{\mathrm{def}}{=}\left\langle a_{1},b_{1},\dots,a_{g},b_{g}:R_{g}\right\rangle ,$an
orientable surface group of genus $g$. 

There is a quotient map $F_{2g}\to\Gamma_{g}$ by reducing modulo
$R_{g}$. An element $w\in F_{2g}$ is a representative of the conjugacy
class of $\gamma\in\Gamma_{g}$ if its image under the quotient map
is in the conjugacy class of $\gamma$ in $\Gamma_{g}.$ We view the
elements $w\in F_{2g}$ as words in the alphabet $\{a_{1}^{\pm1},b_{1}^{\pm1},\dots,a_{g}^{\pm1},b_{g}^{\pm1}\}.$
Then $w$ is cyclically reduced if no two consecutive letters in $w$
are mutual inverses, and the first and last letter are not mutual
inverses (i.e. we do not see $b_{1}b_{1}^{-1}$ for example in any
cyclic permutation of the word). We write $|w|$ or $\ell(w)$ for
the length of the word. We say that $w$ is a shortest representative
of the conjugacy class of $\gamma$ if it has minimal length over
all words with image (under the quotient map) in the conjugacy class
of $\gamma.$ In practice, this means that neither $w$ nor any of
its cyclic permutations can be shortened via free reduction (i.e.
$w$ is cyclically reduced) or via the relation. For example, 
\[
w=ba^{-1}b^{-1}cd\in F_{4}
\]
is cyclically reduced but 
\[
w'=a^{-1}dc\in F_{4}
\]
represents the same conjugacy class in $\Gamma_{g}$, so $w$ is not
a shortest representative its conjugacy class in $\Gamma_{g}.$

\subsection*{Acknowledgements}

We thank Michael Magee and Henry Wilton for their helpful discussions
about this project. Part of the work for this project was carried
out whilst the author was part of a project that received funding
from the European Research Council (ERC) under the European Union\textquoteright s
Horizon 2020 research and innovation programme (grant agreement No
949143).

\section{\label{sec:Combinatorial-integration}Combinatorial integration}

\subsection{\label{subsec:A-combinatorial-formula}A combinatorial formula for
$\protect\integral$}

We set $g=2$ for simplicity of exposition, the proofs extend to any
fixed $g$, and write $\Gamma_{2}=\left\langle a,b,c,d:R_{2}\right\rangle $.
We follow a similar combinatorial integration strategy to that of
the author in \cite{Cassidy2} and Magee in \cite{Magee2021}. The
strategy is to use Proposition \ref{prop: projection formula} to
translate $\chi^{\lambda^{+}(n)}\left(R_{g}(h)\right)$ in to a trace
in $\cnk.$ Indeed, we have that $\mathcal{Q}_{\lambda,n}^{2}=\mathcal{Q}_{\lambda,n}$
and that $\mathcal{Q}_{\lambda,n}$ commutes with the action of $S_{n}$
on $\cnk.$ Thus, we rewrite
\[
\chi^{\lambda^{+}(n)}\left(R_{g}(h)\right)=\frac{1}{d_{\lambda}}\mathrm{Tr}_{\cnk}\left[h_{a}\Q_{\lambda,n}h_{b}\Q_{\lambda,n}\dots h_{c}^{-1}\Q_{\lambda,n}h_{d}^{-1}\Q_{\lambda,n}\right].
\]
We can then compute the expectation of $\integral$ using the Weingarten
calculus for the symmetric group. This enables us to write $\integral$
in a similar fashion to a classical `core graph expansion' (as in
\cite{LinialPuder} for example), from which we derive our bound.
As in \cite{Cassidy2}, the construction of $\left(V^{\lambda^{+}(n)}\right)^{\oplus d_{\lambda}}$
inside $\A\subseteq\cnk$ is crucial. It allows us to refine the usual
Weingarten calculus by exhibiting cancellations that are not immediately
obvious. This is reflected in Lemma \ref{lem: distinctness of vertices}
and Lemma \ref{lem:NO-SINGLETONS}. 

Now, assume $\gamma\neq e$ and $w\in F_{4}$ is a shortest representative
of the conjugacy class of $\gamma.$ Write $w$ as a cyclically reduced
word in the generators 
\begin{equation}
w=x_{1}^{\epsilon_{1}}\dots x_{|w|}^{\epsilon_{|w|}},\label{eq: writing w=00003Dx_1=00005Cdotsx_|w|}
\end{equation}
 with each $x_{i}\in\{a,b,c,d\},$ each $\epsilon_{i}\in\{1,-1\}$
and $|w|=\ell(w)$ the length of the word. For fixed $h=(h_{a},h_{b},h_{c},h_{d})\in S_{n}^{4},$
we can write
\[
\mathrm{Tr}(w(h))=\sum_{i_{j}\in[n]}\left(h_{x_{1}}^{\epsilon_{1}}\right)_{i_{1}i_{2}}\left(h_{x_{2}}^{\epsilon_{2}}\right)_{i_{2}i_{3}}\dots\left(h_{x_{|w|}}^{\epsilon_{|w|}}\right)_{i_{|w|}i_{1}}.
\]
 Now assume $\lambda\vdash k$ with $k\ge0.$ If $k=0,$ then $\chi^{\lambda^{+}(n)}(g)=1$
for all $g\in S_{n}$ (since $V^{\lambda^{+}(n)}$ is the trivial
representation) and in this case the conclusion of Theorem \ref{thm: bound on integral main}
follows from e.g. \cite{Nica,LinialPuder}. So we will assume $k\ge1$
and fix some $\lambda\vdash k.$ Now let $\{v_{p}\}$ be an orthonormal
basis for the subspace $\left(V^{\lambda^{+}(n)}\right)^{\oplus d_{\lambda}}$
inside $\cnk.$ We have 
\[
\begin{aligned}\chi^{\lambda^{+}(n)}\left(R_{2}(h)\right)= & \frac{1}{d_{\lambda}}\sum_{p_{i}}\left\langle h_{a}v_{p_{2}},v_{p_{1}}\right\rangle \left\langle h_{b}v_{p_{3}},v_{p_{2}}\right\rangle \left\langle h_{a}^{-1}v_{p_{4}},v_{3}\right\rangle \left\langle h_{b}^{-1}v_{p_{5}},v_{p_{4}}\right\rangle \\
 & \times\left\langle h_{c}v_{p_{6}},v_{p_{5}}\right\rangle \left\langle h_{d}v_{p_{7}},v_{p_{6}}\right\rangle \left\langle h_{c}^{-1}v_{p_{8}},v_{p_{7}}\right\rangle \left\langle h_{d}^{-1}v_{p_{1}},v_{p_{8}}\right\rangle .
\end{aligned}
\]
Writing each 
\[
v_{p}=\sum_{I}\beta_{p,I}e_{I},
\]
we can rewrite this as 
\[
\begin{aligned} & \chi^{\lambda^{+}(n)}\left(R_{2}(h)\right)\\
= & \frac{1}{d_{\lambda}}\sum_{p_{i}}\sum_{I_{x}^{1},I_{x}^{2},J_{x}^{1},J_{x}^{2}}\left(\beta_{p_{2},I_{a}^{1}}\right)\left(\bar{\beta}_{p_{1},J_{a}^{1}}\right)\left(\beta_{p_{3},I_{b}^{1}}\right)\left(\bar{\beta}_{p_{2},J_{b}^{1}}\right)\left(\beta_{p_{4},J_{b}^{2}}\right)\left(\bar{\beta}_{p_{3},I_{a}^{2}}\right)\left(\beta_{p_{5},J_{a}^{2}}\right)\left(\bar{\beta}_{p_{4},I_{a}^{2}}\right)\\
 & \times\left(\beta_{p_{6},I_{c}^{1}}\right)\left(\bar{\beta}_{p_{5},J_{c}^{1}}\right)\left(\beta_{p_{7},I_{d}^{1}}\right)\left(\bar{\beta}_{p_{6},J_{d}^{1}}\right)\left(\beta_{p_{8},J_{c}^{2}}\right)\left(\bar{\beta}_{p_{7},I_{c}^{2}}\right)\left(\beta_{p_{1},J_{d}^{2}}\right)\left(\bar{\beta}_{p_{8},I_{d}^{2}}\right)\left\langle h_{a}e_{J_{a}^{1}},e_{I_{a}^{1}}\right\rangle \left\langle h_{b}e_{J_{b}^{1}},e_{I_{b}^{1}}\right\rangle \\
 & \times\left\langle h_{a}^{-1}e_{I_{a}^{2}},e_{J_{a}^{2}}\right\rangle \left\langle h_{b}^{-1}e_{I_{b}^{2}},e_{J_{b}^{2}}\right\rangle \left\langle h_{c}e_{J_{c}^{1}},e_{I_{c}^{1}}\right\rangle \left\langle h_{d}e_{J_{d}^{1}},e_{I_{d}^{1}}\right\rangle \left\langle h_{c}^{-1}e_{I_{c}^{2}},e_{J_{c}^{2}}\right\rangle \left\langle h_{d}^{-1}e_{I_{d}^{2}},e_{J_{d}^{2}}\right\rangle .
\end{aligned}
\]

\begin{rem}
\label{rem:distinctness remark}Since $v_{p}\in\left(V^{\lambda^{+}(n)}\right)^{\oplus d_{\lambda}}\subseteq\A,$
we know that $\beta_{p,I}=0$ whenever $I$ is a multi--index with
two indices equal. Thus, we should only consider our sum as being
over multi--indices with all indices distinct. 
\end{rem}

Then, we rewrite the inner product terms as a product of matrix coefficients.
For example, 
\[
\left\langle h_{a}e_{J_{a}^{1}},e_{I_{a}^{1}}\right\rangle \left\langle h_{a}^{-1}e_{I_{a}^{2}},e_{J_{a}^{2}}\right\rangle =\prod_{i=1}^{k}\left(h_{a}\right)_{\left(J_{a}^{1}\right)_{i},\left(I_{a}^{1}\right)_{i}}\left(h_{a}\right)_{\left(J_{a}^{2}\right)_{i},\left(I_{a}^{2}\right)_{i}}
\]
and similarly for $b,c$ and $d$. Denote by 
\[
\begin{aligned}\prod\beta\overset{\mathrm{def}}{=} & \left(\beta_{p_{2},I_{a}^{1}}\right)\left(\bar{\beta}_{p_{1},J_{a}^{1}}\right)\left(\beta_{p_{3},I_{b}^{1}}\right)\left(\bar{\beta}_{p_{2},J_{b}^{1}}\right)\left(\beta_{p_{4},J_{b}^{2}}\right)\left(\bar{\beta}_{p_{3},I_{a}^{2}}\right)\left(\beta_{p_{5},J_{a}^{2}}\right)\left(\bar{\beta}_{p_{4},I_{a}^{2}}\right)\\
 & \times\left(\beta_{p_{6},I_{c}^{1}}\right)\left(\bar{\beta}_{p_{5},J_{c}^{1}}\right)\left(\beta_{p_{7},I_{d}^{1}}\right)\left(\bar{\beta}_{p_{6},J_{d}^{1}}\right)\left(\beta_{p_{8},J_{c}^{2}}\right)\left(\bar{\beta}_{p_{7},I_{c}^{2}}\right)\left(\beta_{p_{1},J_{d}^{2}}\right)\left(\bar{\beta}_{p_{8},I_{d}^{2}}\right).
\end{aligned}
\]
We now have the following expression, for each \emph{fixed }$h=(h_{a},h_{b},h_{c},h_{d})\in S_{n}^{4},$
\[
\begin{aligned} & \chi^{\lambda^{+}(n)}\left(R_{g}(h)\right)\mathrm{Tr}\left(w(h)\right)\\
= & \frac{1}{d_{\lambda}}\sum_{p_{i}}\sum_{I_{x}^{1},I_{x}^{2},J_{x}^{1},J_{x}^{2}}\sum_{i_{j}}\left(\prod\beta\right)\prod_{x\in\{a,b,c,d\}}\left[\prod_{i=1}^{k}\left(h_{x}\right)_{\left(J_{x}^{1}\right)_{i},\left(I_{x}^{1}\right)_{i}}\left(h_{x}\right)_{\left(J_{x}^{2}\right)_{i},\left(I_{x}^{2}\right)_{i}}\right]\left(h_{x_{1}}^{\epsilon_{1}}\right)_{i_{1}i_{2}}\left(h_{x_{2}}^{\epsilon_{2}}\right)_{i_{2}i_{3}}\dots\left(h_{x_{|w|}}^{\epsilon_{|w|}}\right)_{i_{|w|}i_{1}}.
\end{aligned}
\]
We then take the expectation over $S_{n}^{4}.$ For each $x\in\{a,b,c,d\},$
denote by $w_{x}^{+}$ the set of all $i$ for which $x_{i}=x$ and
$\epsilon_{i}=1$ in the expression (\ref{eq: writing w=00003Dx_1=00005Cdotsx_|w|})
for $w$. Similarly, denote by $w_{x}^{-}$ the set of all $i$ for
which $x_{i}=x$ and $\epsilon_{i}=-1.$ Then, for each $x$, we compute
the expectation
\begin{equation}
\int_{S_{n}}\prod_{i=1}^{k}\left(h_{x}\right)_{\left(J_{x}^{1}\right)_{i},\left(I_{x}^{1}\right)_{i}}\left(h_{x}\right)_{\left(J_{x}^{2}\right)_{i},\left(I_{x}^{2}\right)_{i}}\prod_{j\in w_{x}^{+}}\left(h_{x}\right)_{j,j+1}\prod_{\ell\in w_{x}^{-}}\left(h_{x}\right)_{\ell+1,\ell}dh_{x}.\label{eq: integral to bound with weingarten}
\end{equation}
Here, one should actually interpret the $j$ and $\ell$ subscripts
as modulo $|w|,$ so that if either $j$ or $\ell$ are equal to $|w|,$
then $j+1$ or $\ell+1$ are equal to $1$. The integral above can
be computed using the Weingarten calculus. Indeed, using Theorem \ref{thm: Weingarten theorem},
we have that (\ref{eq: integral to bound with weingarten}) is equal
to 
\[
\begin{aligned}\sum_{\sigma_{x},\tau_{x}\in\mathrm{Part}\left(2k+|w_{x}^{+}|+|w_{x}^{-}|\right)} & \sigma_{x}\left(J_{x}^{1}\sqcup J_{x}^{2}\sqcup_{j\in w_{x}^{+}}j\sqcup_{\ell\in w_{x}^{-}}(\ell+1)\right)\tau_{x}\left(I_{x}^{1}\sqcup I_{x}^{2}\sqcup_{j\in w_{x}^{+}}(j+1)\sqcup_{\ell\in w_{x}^{-}}\ell\right)\\
\times & \mathrm{Wg}_{n,\left(2k+|w_{x}^{+}|+|w_{x}^{-}|\right)}\left(\sigma_{x},\tau_{x}\right).
\end{aligned}
\]
One should view the set partitions $\sigma_{x},\tau_{x}$ as partition
diagrams consisting of vertices (one for each vertex label), labeled
from $1$ to $2k+|w_{x}^{+}|+|w_{x}^{-}|$. Then for example $\sigma_{x}\left(J_{x}^{1}\sqcup J_{x}^{2}\sqcup_{j\in w_{x}^{+}}j\sqcup_{\ell\in w_{x}^{-}}(\ell+1)\right)=1$,
whenever two vertices being connected in the partition diagram for
$\sigma_{x}$ implies the corresponding indices are equal. 

Using this, we write
\[
\begin{aligned} & \integral\\
= & \frac{1}{d_{\lambda}}\sum_{p_{i}}\sum_{I_{x}^{1},I_{x}^{2},J_{x}^{1},J_{x}^{2}}\sum_{i_{j}}\left(\prod\beta\right)\prod_{x\in\{a,b,c,d\}}\\
\times\Bigg( & \sum_{\sigma_{x},\tau_{x}\in\mathrm{Part}\left(2k+|w_{x}^{+}|+|w_{x}^{-}|\right)}\sigma_{x}\left(J_{x}^{1}\sqcup J_{x}^{2}\sqcup_{j\in w_{x}^{+}}j\sqcup_{\ell\in w_{x}^{-}}(\ell+1)\right)\tau_{x}\left(I_{x}^{1}\sqcup I_{x}^{2}\sqcup_{j\in w_{x}^{+}}(j+1)\sqcup_{\ell\in w_{x}^{-}}\ell\right)\\
\times & \mathrm{Wg}_{n,\left(2k+|w_{x}^{+}|+|w_{x}^{-}|\right)}\left(\sigma_{x},\tau_{x}\right)\Bigg)
\end{aligned}
\]

We then have the following lemma. 
\begin{lem}
\label{lem: distinctness of vertices}For any $x,$ if any two elements
from $\{1,\dots,k\}$ or $\{k+1,\dots,2k\}$ are in the same block
of $\sigma_{x},$ respectively $\tau_{x}$, then
\[
\sigma_{x}\left(J_{x}^{1}\sqcup J_{x}^{2}\sqcup_{j\in w_{x}^{+}}j\sqcup_{\ell\in w_{x}^{-}}(\ell+1)\right)=0,
\]
respectively, $\tau_{x}\left(I_{x}^{1}\sqcup I_{x}^{2}\sqcup_{j\in w_{x}^{+}}(j+1)\sqcup_{\ell\in w_{x}^{-}}\ell\right)=0.$
\end{lem}

\begin{proof}
Suppose $\alpha,\beta\in\{1,\dots,k\}$ and $\alpha\sim\beta$ in
$\sigma_{x}$. Then, 
\[
\sigma_{x}\left(J_{x}^{1}\sqcup J_{x}^{2}\sqcup_{j\in w_{x}^{+}}j\sqcup_{\ell\in w_{x}^{-}}(\ell+1)\right)=0
\]
 unless $(J_{x}^{1})_{\alpha}=(J_{x}^{1})_{\beta}$. But this is not
possible by Remark \ref{rem:distinctness remark}. This obviously
extends to $k+\alpha,k+\beta\in\{k+1,\dots,2k\}$ and also to $\tau_{x}.$ 
\end{proof}
This is the first stage at which our construction of the space $\A$
allows us to see some cancellations, since we use the observation
that $\left(V^{\lambda^{+}(n)}\right)^{\oplus d_{\lambda}}\subseteq\left\langle e_{i_{1}}\otimes\dots\otimes e_{i_{k}}:\ i_{1},\dots,i_{k}\ \mathrm{all\ distinct}\right\rangle .$
To capture more cancellations, it is then sensible to use the other
condition on the subspace $\A,$ namely that 
\[
\A\subseteq\bigcap_{j=1}^{k}\ker\left(T_{j}\right).
\]
The presence of the vertices in $\sigma_{x}$ and $\tau_{x}$ that
correspond to indices coming from $\mathrm{Tr}\left(w(h)\right)$
complicates proceedings slightly as compared to \cite[Lemma 4.4]{Cassidy2},
but the refinement needed is described in the following lemma. Given
a collection of partitions $\sigma_{x},\tau_{x}$ for $x\in\{a,b,c,d\}$,
their meet
\[
\bigsqcup_{x}\sigma_{x}\sqcup\tau_{x}
\]
 defines a set partition of the set of index labels 
\[
\big\{\left(J_{x}^{j}\right)_{i},\left(I_{x}^{j}\right)_{i},i_{1},\dots,i_{|w|}:\ x\in\{a,b,c,d\},j=1,2,i\in[k]\big\}.
\]
We denote this partition $\mathcal{P}(\sigma_{x},\tau_{x}).$\footnote{This is almost equivalent to the partition of the vertices in the
graph expansion described in \ref{subsec:Construction-of} -- it
is missing only the additional $\pi_{i}$ identifications, which are
yet to be introduced. }
\begin{lem}
\label{lem:NO-SINGLETONS}For any fixed collection of $\sigma_{x},\tau_{x}$
for $x\in\{a,b,c,d\},$ if a vertex label is a singleton in $\mathcal{P}(\sigma_{x},\tau_{x})$,
then 
\[
\begin{aligned} & \frac{1}{d_{\lambda}}\sum_{p_{i}}\sum_{I_{x}^{1},I_{x}^{2},J_{x}^{1},J_{x}^{2}}\sum_{i_{j}}\left(\prod\beta\right)\prod_{x\in\{a,b,c,d\}}\\
\times & \sigma_{x}\left(J_{x}^{1}\sqcup J_{x}^{2}\sqcup_{j\in w_{x}^{+}}j\sqcup_{\ell\in w_{x}^{-}}(\ell+1)\right)\tau_{x}\left(I_{x}^{1}\sqcup I_{x}^{2}\sqcup_{j\in w_{x}^{+}}(j+1)\sqcup_{\ell\in w_{x}^{-}}\ell\right)\\
\times & \mathrm{Wg}_{n,\left(2k+|w_{x}^{+}|+|w_{x}^{-}|\right)}\left(\sigma_{x},\tau_{x}\right)=0.
\end{aligned}
\]
\end{lem}

\begin{proof}
WLOG, assume that the label $\left(J_{a}^{1}\right)_{1}$ is a singleton
in $\mathcal{P}(\sigma_{x},\tau_{x}).$ Then, for all other variables
fixed except for $\left(J_{a}^{1}\right)_{1}$ , the contribution
to the above sum is
\[
\sum_{\left(J_{a}^{1}\right)=1}^{n}\left(\bar{\beta}_{p_{1},J_{a}^{1}}\right),
\]
multiplied by some constant coming from the other fixed terms. This
is exactly the (complex conjugate of) the coefficient of $e_{\left(J_{a}^{1}\right)_{2}}\otimes\dots\otimes e_{\left(J_{a}^{1}\right)_{k}}$
in $T_{1}(v_{p_{1}}).$ Since $v_{p_{1}}$ belongs to an orthonormal
basis for $\mathcal{U}_{\lambda^{+}(n)},$ we have that $v_{p_{1}}\in\ker\left(T_{1}\right)$,
so that $\sum_{\left(J_{a}^{1}\right)=1}^{n}\left(\bar{\beta}_{p_{1},J_{a}^{1}}\right)=0.$ 
\end{proof}
We will say that $\mathcal{P}(\sigma_{x},\tau_{x})$ is singleton
free if no vertex label is a singleton, and write 
\[
\sum_{\sigma_{x},\tau_{x}}^{*}
\]
to indicate that we are summing over singleton free collections of
partitions. Now we can use our projection formula to further simplify
our expression. We rewrite 
\[
\begin{aligned}\sum_{p_{2}}\left(\beta_{p_{2},I_{a}^{1}}\right)\left(\bar{\beta}_{p_{2},J_{b}^{1}}\right) & =\sum_{p_{2}}\left\langle e_{I_{a}^{1}},v_{p_{2}}\right\rangle \left\langle v_{p_{2}},e_{J_{b}^{1}}\right\rangle \\
 & =\left\langle \mathcal{Q}_{\lambda,n}\left(e_{I_{a}^{1}}\right),e_{J_{b}^{2}}\right\rangle \\
 & =\sum_{\pi}^{\le S_{k}}c(n,k,\lambda,\pi)\left\langle P_{\pi}e_{I_{a}^{1}},e_{J_{b}^{2}}\right\rangle .
\end{aligned}
\]
We repeat this with $p_{1},\dots,p_{8}$ and the corresponding $\beta$
products. For each fixed collection of multi--indices $I_{a}^{1},I_{a}^{2},\dots,J_{d}^{1},J_{d}^{2},$
we can thus rewrite 
\[
\begin{aligned} & \sum_{p_{i}}\left(\prod\beta\right)\\
= & \sum_{\pi_{1},\dots,\pi_{8}}^{\le S_{k}}\left(\prod_{i=1}^{8}c(n,k,\lambda,\pi_{i})\right)\left\langle P_{\pi_{1}}e_{J_{d}^{2}},e_{J_{a}^{1}}\right\rangle \left\langle P_{\pi_{2}}e_{I_{a}^{1}},e_{J_{b}^{1}}\right\rangle \left\langle P_{\pi_{3}}e_{I_{b}^{1}},e_{I_{a}^{2}}\right\rangle \\
 & \times\left\langle P_{\pi_{4}}e_{J_{a}^{2}},e_{I_{b}^{1}}\right\rangle \left\langle P_{\pi_{5}}e_{J_{b}^{2}},e_{J_{c}^{1}}\right\rangle \left\langle P_{\pi_{6}}e_{I_{c}^{1}},e_{J_{d}^{1}}\right\rangle \left\langle P_{\pi_{7}}e_{I_{d}^{1}},e_{I_{c}^{2}}\right\rangle \left\langle P_{\pi_{8}}e_{J_{c}^{2}},e_{I_{d}^{2}}\right\rangle .
\end{aligned}
\]
Putting this all together gives 
\[
\begin{aligned}\integral & =\frac{1}{d_{\lambda}}\sum_{\sigma_{x},\tau_{x}}^{*}\sum_{\pi_{1},\dots,\pi_{8}}^{\le S_{k}}\sum_{I_{x}^{1},I_{x}^{2},J_{x}^{1},J_{x}^{2}}\left(\prod_{i=1}^{8}c(n,k,\lambda,\pi_{i})\right)\left(\prod_{x\in\{a,b,c,d\}}\mathrm{Wg}_{n,\left(2k+|w_{x}^{+}|+|w_{x}^{-}|\right)}\left(\sigma_{x},\tau_{x}\right)\right)\\
 & \times\left(\prod_{x\in\{a,b,c,d\}}\sigma_{x}\left(J_{x}^{1}\sqcup J_{x}^{2}\sqcup_{j\in w_{x}^{+}}j\sqcup_{\ell\in w_{x}^{-}}(\ell+1)\right)\tau_{x}\left(I_{x}^{1}\sqcup I_{x}^{2}\sqcup_{j\in w_{x}^{+}}(j+1)\sqcup_{\ell\in w_{x}^{-}}\ell\right)\right)\\
 & \times\left\langle P_{\pi_{1}}e_{J_{d}^{2}},e_{J_{a}^{1}}\right\rangle \left\langle P_{\pi_{2}}e_{I_{a}^{1}},e_{J_{b}^{1}}\right\rangle \left\langle P_{\pi_{3}}e_{I_{b}^{1}},e_{I_{a}^{2}}\right\rangle \left\langle P_{\pi_{4}}e_{J_{a}^{2}},e_{I_{b}^{1}}\right\rangle \left\langle P_{\pi_{5}}e_{J_{b}^{2}},e_{J_{c}^{1}}\right\rangle \\
 & \times\left\langle P_{\pi_{6}}e_{I_{c}^{1}},e_{J_{d}^{1}}\right\rangle \left\langle P_{\pi_{7}}e_{I_{d}^{1}},e_{I_{c}^{2}}\right\rangle \left\langle P_{\pi_{8}}e_{J_{c}^{2}},e_{I_{d}^{2}}\right\rangle .
\end{aligned}
\]
This formula asserts that we can evaluate $\integral$ by summing
over some collection of set partitions and then counting how many
multi--indices and indices satisfy the constraints imposed. Just
as in \cite[Definition 3.5]{Magee2021} and similar to \cite[Definition 4.5]{Cassidy2},
we formalize each collection of $\sigma_{x},\tau_{x}$ and $\pi_{i}$
as a matching datum. 
\begin{defn}
A matching datum for $(w,k)$ is a collection of set partitions $\sigma_{x},\tau_{x}\in\mathrm{Part}\left(2k+|w_{x}^{+}|+|w_{x}^{-}|\right)$
for each $x\in\{a,b,c,d\}$ with $\mathcal{P}(\sigma_{x},\tau_{x})$
singleton free, together with a collection of set partitions $\pi_{1},\dots,\pi_{8}\in\skpart.$
We will write 
\[
\mathrm{Match}\left(w,k\right)
\]
for the (finite) collection of these data and we will write 
\[
\mathcal{N}\left(\sigma_{x},\tau_{x},\pi_{i}\right)
\]
for the total number of multi--indices and indices, $I_{a}^{1},I_{a}^{2},\dots,J_{d}^{1},J_{d}^{2},i_{1},\dots,i_{|w|}$,
satisfying 
\end{defn}

\begin{itemize}
\item $\sigma_{x}^{\mathrm{weak}}\left(J_{x}^{1}\sqcup J_{x}^{2}\sqcup_{j\in w_{x}^{+}}j\sqcup_{\ell\in w_{x}^{-}}(\ell+1)\right)=1$
for each $x\in\{a,b,c,d\}$,
\item $\tau_{x}^{\mathrm{weak}}\left(I_{x}^{1}\sqcup I_{x}^{2}\sqcup_{j\in w_{x}^{+}}(j+1)\sqcup_{\ell\in w_{x}^{-}}\ell\right)=1$
for each $x\in\{a,b,c,d\}$ and 
\item $\left\langle P_{\pi_{1}}e_{J_{d}^{2}},e_{J_{a}^{1}}\right\rangle =\left\langle P_{\pi_{2}}e_{I_{a}^{1}},e_{J_{b}^{1}}\right\rangle =\left\langle P_{\pi_{3}}e_{I_{b}^{1}},e_{I_{a}^{2}}\right\rangle =\left\langle P_{\pi_{4}}e_{J_{a}^{2}},e_{I_{b}^{1}}\right\rangle =\left\langle P_{\pi_{5}}e_{J_{b}^{2}},e_{J_{c}^{1}}\right\rangle =\left\langle P_{\pi_{6}}e_{I_{c}^{1}},e_{J_{d}^{1}}\right\rangle =\left\langle P_{\pi_{7}}e_{I_{d}^{1}},e_{I_{c}^{2}}\right\rangle =\left\langle P_{\pi_{8}}e_{J_{c}^{2}},e_{I_{d}^{2}}\right\rangle =1.$ 
\end{itemize}
In this notation, we have the following theorem giving a combinatorial
expression for $\integral.$ 
\begin{thm}
\label{thm: combinatorial expression for itnegral}For any $k\ge1$
and any $\lambda\vdash k$ and $w\in F_{4}$ with $w\neq e,$ we have
\[
\integral=\frac{1}{d_{\lambda}}\sum_{(\sigma_{x},\tau_{x},\pi_{i})\in\mathrm{Match}\left(w,k\right)}\left(\prod_{i=1}^{8}c(n,k,\lambda,\pi_{i})\right)\left(\prod_{x\in\{a,b,c,d\}}\mathrm{Wg}_{n,\left(2k+|w_{x}^{+}|+|w_{x}^{-}|\right)}\left(\sigma_{x},\tau_{x}\right)\right)\mathcal{N}\left(\sigma_{x},\tau_{x},\pi_{i}\right).
\]
\end{thm}

By the observation in \S\ref{subsec:The-Weingarten-calculus}, we
have that 
\[
\mathrm{Wg}_{n,\left(2k+|w_{x}^{+}|+|w_{x}^{-}|\right)}\left(\sigma_{x},\tau_{x}\right)=O\left(n^{-|\sigma_{x}\wedge\tau_{x}|}\right)
\]
and in (\ref{eq: coeff in projection is O(-del)}) we asserted that
\[
c(n,k,\lambda,\pi_{i})=O\left(n^{-\mathrm{del}(\pi_{i})}\right).
\]
We thus deduce the following proposition from Theorem \ref{thm: combinatorial expression for itnegral}. 
\begin{prop}
\label{prop: bound on integral in terms of del, sigma, tau, N}We
have 
\[
\integral\ll_{k,\ell(w)}\sum_{(\sigma_{x},\tau_{x},\pi_{i})\in\mathrm{Match}\left(w,k\right)}n^{-\sum_{i}\mathrm{del}(\pi_{i})}n^{-\sum_{f}|\sigma_{x}\wedge\tau_{x}|}\mathcal{N}\left(\sigma_{x},\tau_{x},\pi_{i}\right).
\]
\end{prop}

\subsection{\label{subsec:A-core-graph}A core graph type expansion}

To obtain the bound in Theorem \ref{thm: bound on integral main}
from Proposition \ref{prop: bound on integral in terms of del, sigma, tau, N}
we now describe how one can construct a graph $\gamsig$ from a given
$(\sigma_{x},\tau_{x},\pi_{i})\in\mathrm{Match}\left(w,k\right).$
Such a graph is similar in spirit to a core graph (as in \cite{Stallings})
or more specifically, a quotient of a (not necessarily connected)
core graph corresponding to subgroups of $F_{4}$ generated by powers
of $R_{2}$. The presence of a $w$--cycle and the fact that, in
general, $\sigma_{x}\neq\tau_{x}$ and $\pi_{i}\notin\iota\left(S_{k}\right),$
is the main difference between $\gamsig$ and the usual notion of
a (not necessarily connected) core graph. In any case, an asymptotic
bound on each summand will be encoded in the Euler characteristic
of the graph, 
\[
\chi\left(\gamsig\right)=|V\left(\gamsig\right)|-|E\left(\gamsig\right)|.
\]
Our strategy will then be to bound the maximum $\chi$ of a graph
$\gamsig$ over all $(\sigma_{x},\tau_{x},\pi_{i})\in\match.$

\subsubsection{\label{subsec:Construction-of}Construction of $\protect\gamsig$}

We first construct the $R$--cycles. Here, we again take $R=R_{2}$
for simplicity, but the construction remains the same for any fixed
$g.$ For each $x,$ for each of the $2k$ pairs 
\[
\left(J_{x}^{j}\right)_{i},\left(I_{x}^{j}\right)_{i}
\]
with $i\in[k]$ and $j\in\{1,2\},$ we have two vertices labeled by
$\left(J_{x}^{j}\right)_{i},\left(I_{x}^{j}\right)_{i}$ respectively
and add a directed $x$--labeled edge from $\left(J_{x}^{j}\right)_{i}$
to $\left(I_{x}^{j}\right)_{i}.$ 

Now we construct the $w$--cycle. We define a vertex for each $i_{1},\dots,i_{|w|}$
(with the respective labels) and draw a directed $x_{j}$--labeled
edge from $i_{j}$ and $i_{j+1},$ Unless $j=|w|,$ in which case
we draw an $x_{|w|}$--labeled edge from $i_{|w|}$to $i_{1}.$ This
forms a cycle that we orient to read $w$ as we traverse the edges
clockwise. See Figure \ref{fig: R cycles and w cycles} below. 

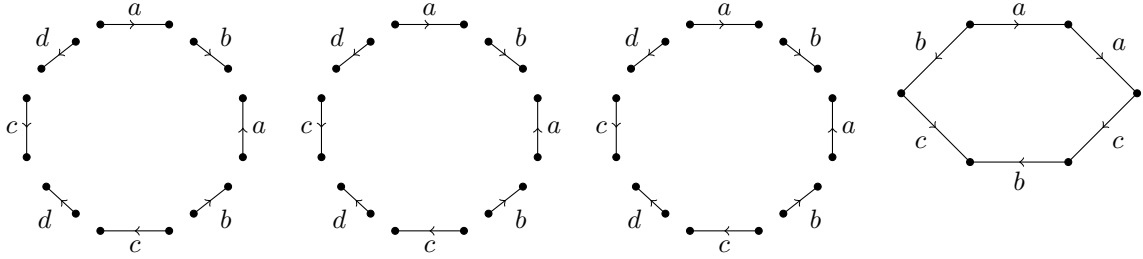
\begin{figure}[h]
\centering\begin{tikzpicture}[scale=0.65, baseline=(current bounding box.center)]

\myoctagon{A}
\begin{scope}[shift={(6,0)}] \myoctagon{B} \end{scope}
\begin{scope}[shift={(12,0)}] \myoctagon{C} \end{scope}

\begin{scope}[shift={(18,0)}] \gammaloop \end{scope}

\end{tikzpicture}

\caption{\label{fig: R cycles and w cycles}Here we have $k=3,$ so we have
three $R_{2}$--cycles on the left and the $w$--cycle on the right,
with $w=a^{2}cbc^{-1}b^{-1}.$ We have omitted the vertex labels. }
\end{figure}

\subsubsection*{Adding $\pi_{i}$--arcs}

Now fix some $(\sigma_{x},\tau_{x},\pi_{i})\in\match.$ We will add
arcs to our construction according to the identifications of indices
determined by these partitions. For example, for $\pi_{1},$ we have
\[
\left\langle P_{\pi_{1}}e_{J_{d}^{2}},e_{J_{a}^{1}}\right\rangle =1
\]
if and only if, for $i,j\in[k]$, 
\[
j\sim k+i\ \mathrm{in}\ \pi_{1}\iff\left(J_{d}^{2}\right)_{j}=\left(J_{a}^{1}\right)_{i}.
\]
So we draw a $\pi_{1}$--arc between these vertices. We do this for
every block of $\pi_{1},$ drawing a total of $2k-|\pi_{1}|$ arcs.
We repeat this process with $\pi_{2}$ and $I_{a}^{1},J_{b}^{1},$
$\pi_{3}$ and $I_{b}^{1},I_{a}^{2}$, $\dots,$ $\pi_{8}$ and $J_{c}^{2},I_{d}^{2}.$
In this way, the $\pi_{i}$ determine an equivalence relation on the
vertices in the $R$--cycles, whereby $u\sim v$ iff there is a $\pi_{i}$--arc
between $u$ and $v$. 

\subsubsection*{Adding $\sigma_{x},\tau_{x}$--arcs}

For each $x,$ we have 
\[
\sigma_{x}^{\mathrm{weak}}\left(J_{x}^{1}\sqcup J_{x}^{2}\sqcup_{j\in w_{x}^{+}}i_{j}\sqcup_{\ell\in w_{x}^{-}}i_{(\ell+1)}\right)=1
\]
 if whenever two vertices are connected by an edge in the $\sigma_{x}$
partition diagram, then the corresponding index labels in $J_{x}^{1}\sqcup J_{x}^{2}\sqcup_{j\in w_{x}^{+}}i_{j}\sqcup_{\ell\in w_{x}^{-}}i_{(\ell+1)}$
are equal. Thus, for any edge in the $\sigma_{x}$ diagram, we draw
a $\sigma_{x}$--arc between the corresponding vertices in our cycles.
We do the same for $\tau_{x}.$ In this way, $\sigma_{x}$ determines
an equivalence relation on the set of\emph{ initial }vertices of $x$--labeled
edges in both the $R$--cycles and the $w$--cycle, whereby $u\sim v$
iff there is a $\sigma_{x}$--arc between $u$ and $v$. Similarly,
$\tau_{x}$ determines an equivalence relation on the set of \emph{terminal
}vertices of $x$--labeled edges. 

\subsubsection*{Quotienting by $(\sigma_{x},\tau_{x},\pi_{i})$--arcs}

There is now an equivalence relation on the set of all vertices, 
\[
\big\{\left(J_{x}^{j}\right)_{i},\left(I_{x}^{j}\right)_{i},i_{1},\dots,i_{|w|}:\ x\in\{a,b,c,d\},j=1,2,i\in[k]\big\}
\]
whereby $u\sim v$ iff there are connected by some combination of
$\sigma_{x},\tau_{x}$ or $\pi_{i}$--arcs. Equivalently, this defines
some partition $\mathcal{P}_{(\sigma_{x},\tau_{x},\pi_{i})}$ on the
above set. We then glue the vertices together according to this partition.
That is, any vertices that are in the same block of $\mathcal{P}_{(\sigma_{x},\tau_{x},\pi_{i})}$
are glued together to form one vertex. These form the vertices of
$\gamsig.$ 

Initially, for each $x,$ we have $2k$ copies of an $x$--labeled
edge in the $R$--loops and and additional number of copies in the
$w$--loop. To determine the edges of $\gamsig,$ we fold $x$--labeled
edges together in the following way: if there is a $\sigma_{x}$--arc
between the initial vertices of two $x$--labeled edges \emph{and
}a $\tau_{x}$--arc between their terminal vertices (implying that
their initial/terminal vertices belong to the same block of $\mathcal{P}_{(\sigma_{x},\tau_{x},\pi_{i})}$),
then we fold these two edges in to one $x$--labeled edge in $\gamsig.$
Note that Lemma \ref{lem: distinctness of vertices} determines that
there is never a $\sigma_{x}$--arc from $\left(J_{x}^{j}\right)_{i}$
to $\left(J_{x}^{j}\right)_{\ell}$ and similarly for $\tau_{x}.$
We then have the following lemmas relating the number of vertices
and edges of $\gamsig$ to Proposition \ref{prop: bound on integral in terms of del, sigma, tau, N}.
\begin{lem}
$N(\sigma_{x},\tau_{x},\pi_{i})\le n^{|V(\gamsig)|}.$
\end{lem}

\begin{proof}
Each vertex represents an index whose values are all forced to be
equal by the $\sigma_{x},\tau_{x}$ and $\pi_{i}.$ 
\end{proof}
\begin{lem}
$n^{-\sum_{f}|\sigma_{x}\wedge\tau_{x}|}=n^{-|E(\gamsig)|}.$
\end{lem}

\begin{proof}
Clearly, for each fixed $x$, each connected component of $\sigma_{x}\wedge\tau_{x}$
corresponds to an $x$--labeled edge in $\gamsig.$ 
\end{proof}
Altogether this yields the following proposition. 
\begin{prop}
\label{prop:bound on integral del and chi}For any $k\ge1,$ $\lambda\vdash k$
and any $w\in F_{4}$ with $w\neq e,$ we have 
\[
\integral\ll_{k,\ell(w)}\sum_{(\sigma_{x},\tau_{x},\pi_{i})\in\match}n^{-\sum_{i}\mathrm{del}(\pi_{i})}n^{\chi\left(\gamsig\right)}.
\]
\end{prop}

\begin{rem}
The significance of Lemma \ref{lem:NO-SINGLETONS} and Lemma \ref{lem: distinctness of vertices}
is the following. In constructing $\gamsig,$ without Lemma \ref{lem: distinctness of vertices},
one could take $\pi_{i}=\mathrm{Id}$ for each $i$, glue $\left(J_{x}^{1}\right)_{1},\dots,\left(J_{x}^{1}\right)_{k}$
together and $\left(J_{x}^{2}\right)_{1},\dots,\left(J_{x}^{2}\right)_{k}$
together via $\sigma_{x}$--arcs (doing the same with the $\tau_{x}$--arcs)
and fold the respective $x$--labeled edges together in to two distinct
$x$--labeled edges. Informally, this would correspond to gluing
each $R$--cycle (each reading exactly $R$ when traversed) on top
of each other and leaving the $w$--cycle untouched. The result would
have Euler characteristic $\chi=0$ and the bound given in Theorem
\ref{thm: bound on integral main} would not be derived from Proposition
\ref{prop: bound on integral in terms of del, sigma, tau, N}. Without
Lemma \ref{lem:NO-SINGLETONS}, one could again take $\pi_{i}=\mathrm{Id}$
for each $i$ and not add any $\sigma_{x}$--arcs or $\tau_{x}$--arcs.
This would correspond to leaving each of the $k$ $R$--cycles unglued
and the resulting $\gamsig$ would also have Euler characteristic
$\chi=0.$ These lemmas are therefore crucial for this argument to
succeed and it is the explicit construction of $\left(V^{\lambda^{+}(n)}\right)^{\oplus d_{\lambda}}$
inside $\A\subseteq\cnk$ that allows for this.
\end{rem}

\subsubsection{\label{subsec:Simplifying-lemmas}Simplifying lemmas}

The following lemmas simplify our expression by allowing us to take
the $\pi_{i}$ as permutations and $\sigma_{x}=\tau_{x}.$ 
\begin{lem}
Suppose each $\pi_{i}\le\hat{\pi}_{i}$ where $\hat{\pi}_{i}\in\iota\left(S_{k}\right)$.
Then 
\[
n^{-\sum_{i}\mathrm{del}(\pi_{i})}n^{\chi\left(\gamsig\right)}\ll n^{\chi\left(\Gamma(\sigma_{x},\tau_{x},\hat{\pi}_{i}\right)}.
\]
\end{lem}

\begin{proof}
Suppose we add $P$ edges to complete the $\pi_{i}$ to $\hat{\pi}_{i},$
i.e. 
\[
\sum_{i}\mathrm{del}(\pi_{i})=P.
\]

There are two possible situations that arise each time we add an edge
to say $\pi_{j}.$ Either this edge, which defines an additional arc,
glues together two distinct vertices of $\gamsig$ (thus \emph{decreasing
}$\chi$ by $1$) or the edge forces a gluing of a vertex to itself
(i.e. the endpoints of the additional $\pi_{j}$--arc are already
connected by some combination of $\sigma_{x},\tau_{x}$ and $\pi_{i}$--arcs
and are thus already glued together) in which case $\chi$ remains
unchanged. So clearly $\chi$ decreases by at most the number of edges
added to the collection of $\pi_{i}$, which is $P$, so we have 
\[
\chi\left(\gamsig\right)\le\chi\left(\Gamma\left(\sigma_{x},\tau_{x},\hat{\pi}_{i}\right)\right)+P.
\]
 
\end{proof}
\begin{lem}
We have 
\[
n^{\chi\left(\gamsig\right)}\ll\max\bigg\{ n^{\chi\left(\Gamma(\sigma_{x},\sigma_{x},\pi_{i}\right)},n^{\chi\left(\Gamma(\tau_{x},\tau_{x},\pi_{i}\right)}\bigg\}.
\]
\end{lem}

\begin{proof}
Without loss of generality, let $x=a$ and assume $\sigma_{x'}$ and
$\tau_{x'}$ are all fixed for $x'\neq x$. Now fix some $\sigma_{x}$
and $\tau_{x}.$ The number of $x$--labeled edges in $\gamsig$
is exactly $|\sigma_{x}\wedge\tau_{x}|$. Consider first undoing the
$\tau_{x}$ identifications that are different from $\sigma_{x}$
until the only arcs between the terminal vertices of $x$--labeled
edges are those present in $\sigma_{x}\wedge\tau_{x}.$ With each
arc we remove, the number of edges does not change and the number
of vertices in our graph either stays the same or increases by one. 

Then we add identifications (to the terminal vertices of the $x$--labeled
edges) so that they match $\sigma_{x}.$ Each time we add an identification,
we decrease the number of $x$--labeled edges by one, and the total
number of vertices either stays the same or decreases by one and hence,
in total, $\chi$ either stays the same or increases.
\end{proof}
We will now denote by 
\[
\mathrm{Match}_{S_{k}}(w,k)\subset\match
\]
the set of all matching data $(\sigma_{x},\sigma_{x},\pi_{i})\in\match$
where $\pi_{i}$ is the image of some permutation in $\part$. We
will abbreviate our notation and write $(\sigma_{x},\pi_{i})$ for
the elements of $\mathrm{Match}_{S_{k}}(w,k).$ The previous two lemmas
and Proposition \ref{prop:bound on integral del and chi} combined
give the following corollary. 
\begin{cor}
\label{cor: INTEGRAL bounded by n^=00005Cchi}For any $k\ge1,$ $\lambda\vdash k$
and any $w\in F_{4}$ with $w\neq e,$ we have 
\[
\integral\ll_{k,\ell(w)}n^{\max\chi\left(\samgamsig\right)},
\]
where the maximum is taken over all $(\sigma_{x},\pi_{i})\in\mathrm{Match}_{S_{k}}(w,k).$ 
\end{cor}

Thus, our task reduces to bounding 
\[
\max_{(\sigma_{x},\pi_{i})\in\mathrm{Match}_{S_{k}}(w,k)}\chi\left(\samgamsig\right).
\]
We will prove the following proposition. 
\begin{prop}
\label{prop:bounding euler char of Gamma -k}For any $k$ fixed, if
$w$ is a shortest representative of the conjugacy class of $\gamma\in\Gamma_{g}$
with $\gamma\neq e,$ then 
\[
\max_{(\sigma_{x},\pi_{i})\in\mathrm{Match}_{S_{k}}(w,k)}\chi\left(\samgamsig\right)\le-k.
\]
\end{prop}

It is clear that Corollary \ref{cor: INTEGRAL bounded by n^=00005Cchi}
combined with Proposition \ref{prop:bounding euler char of Gamma -k}
immediately imply Theorem \ref{thm: bound on integral main}. The
rest of this paper is then dedicated to proving Proposition \ref{prop:bounding euler char of Gamma -k}.

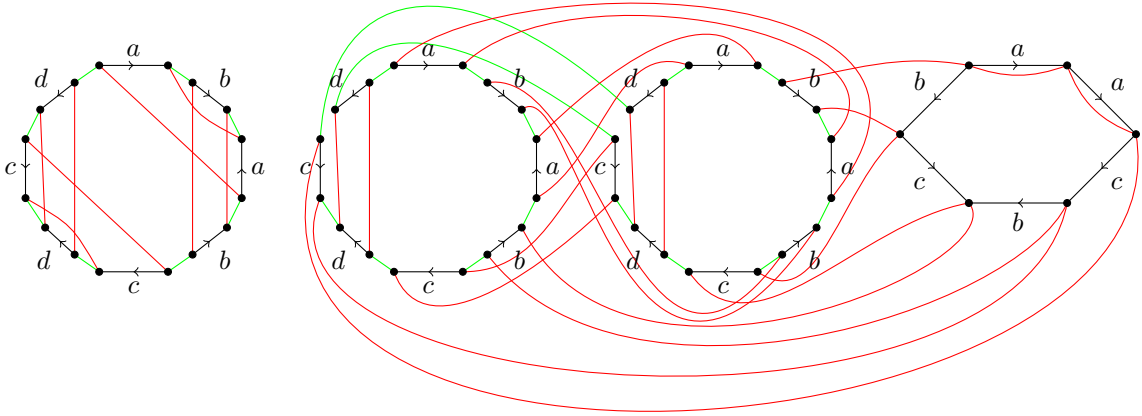
\begin{figure}[H]
\centering\begin{tikzpicture}[scale=0.65, baseline=(current bounding box.center)]

\pgfdeclarelayer{bg}
\pgfsetlayers{bg,main}
\myoctagon{A}
\begin{scope}[shift={(6,0)}] \myoctagon{B} \end{scope}
\begin{scope}[shift={(12,0)}] \myoctagon{C} \end{scope}

\begin{scope}[shift={(18,0)}] \gammaloop \end{scope}
\begin{pgfonlayer}{bg}
\draw[green] (Au1) -- (Av2);\draw[green] (Au2) -- (Av3);\draw[green] (Au3) -- (Av4);\draw[green] (Au4) -- (Av5);\draw[green] (Au5) -- (Av6);\draw[green] (Au6) -- (Av7);\draw[green] (Au7) -- (Av8);\draw[green] (Au8) -- (Av1);
\draw[green] (Bu1) -- (Bv2);\draw[green] (Bu2) -- (Bv3);\draw[green] (Bu3) -- (Bv4);\draw[green] (Bu4) -- (Bv5);\draw[green] (Bu5) -- (Bv6);\draw[green] (Bu6) -- (Bv7);\draw[green] (Bu7) ..controls (5, 5) and (8,5).. (Cv8);\draw[green] (Bu8) -- (Bv1);
\draw[green] (Cu1) -- (Cv2);\draw[green] (Cu2) -- (Cv3);\draw[green] (Cu3) -- (Cv4);\draw[green] (Cu4) -- (Cv5);\draw[green] (Cu5) -- (Cv6);\draw[green] (Cu6) -- (Cv7);\draw[green] (Bv8) .. controls (5.5, 4) and (7.5, 4)..(Cu7);\draw[green] (Cu8) -- (Cv1);
\draw[red] (Av1)--(Au3);\draw[red] (Au1)..controls(2.25, 2) and (2.25, 2)..(Av3);\draw[red] (Av2)--(Au4);\draw[red] (Au2)--(Av4);\draw[red] (Av5)--(Au7);\draw[red] (Au5)..controls (-0.25,-0.25) and (-0.25,-0.25)..(Av7);\draw[red] (Av6)--(Au8);\draw[red] (Au6)--(Av8);
\draw[red] (Bv6)--(Bu8); \draw[red] (Bv8)--(Bu6);
\draw[red] (Cv6)--(Cu8); \draw[red] (Cv8)--(Cu6);\draw[red] (Bv5) to [out=0, in=225](Cu7); \draw[red] (Bu5) to [out = 290, in=225] (Cv7);\draw[red] (Bv1)..controls (8,5) and (19,5)..  (Cu3); \draw[red] (Bu1) to [out=45, in=45] (Cv3);
\draw[red] (Cv1) to[out=160, in=30] (Bu3); \draw[red] (Cu1) to [out=120, in=45] (Bv3);\draw[red] (Bv2) ..controls (10.5,3) and (11.2, -5).. (Cu4); \draw[red] (Bu2) ..controls (10,3) and (11, -6)..(Cv4);
\draw[red] (Cv2) to[out =10, in=170] (1); \draw[red] (Cu2) to[out=10, in=150] (6);
\draw[red] (Cv5) to [out=325, in=220] (6); \draw[red] (Cu5) to [out =305, in=190] (5);
\draw[red] (Bv4) ..controls (11,-4.5) and (19, -1).. (5); \draw[red] (Bu4)..controls (11,-4.5) and (19, -2).. (4);
\draw[red] (Bv7)..controls (3,-4) and (19, -5).. (4); \draw[red] (Bu7)..controls (1.75,-6.75) and (22.5, -5).. (3); 
\draw[red] (1) to [out = 340, in=200] (2); \draw[red] (2) to [out = 290, in =160] (3);
\end{pgfonlayer}

\end{tikzpicture}

\vspace{-1.5cm}
\caption{\label{fig:Gamma before quotient}Above are the $R$--cycles and
$w$--cycle from Figure \ref{fig: R cycles and w cycles}, together
with some $\sigma_{x}$--arcs and $\tau_{x}$--arcs (in red, with
$\sigma_{x}=\tau_{x}$) and $\pi_{i}$--arcs (in green, with each
$\pi\in\iota\left(\protect\csk\right)$). The corresponding $\protect\samgamsig$
is in Figure \ref{fig:Gamma after quotient}. }
\end{figure}

\begin{figure}[H]
\centering\begin{tikzpicture}[scale=0.65, baseline=(current bounding box.center)]

    \coordinate (O) at (0,0);
    \draw[->-] (O) 
        .. controls (-2,2) and (-2,-2) .. (O) 
        node[midway, left] {$a$};
    \draw[->-] (O) 
        .. controls (-2,-2) and (2,-2) .. (O) 
        node[midway, below] {$b$};
    \draw[->-] (O) 
        .. controls (2,-2) and (2,2) .. (O) 
        node[midway, right] {$c$};
    \draw[->-] (O) 
        .. controls (2,2) and (-2,2) .. (O) 
        node[midway, above] {$d$};
    \filldraw (O) circle (2pt);

\coordinate (1) at (6,3); \coordinate (2) at (9,0); \coordinate (3) at (6,-3); \coordinate (4) at (3,0); 
\draw[->-] (1)--(2) node[midway, below left]{$c$};
\draw[->-] (2)--(3)node[midway, below right]{$b$};
\draw[->-] (4)--(3)node[midway, below left]{$c$};
\draw[->-] (1)--(4)node[midway, below right]{$b$};

\draw[->-] (3) to [out =120, in=240] node[midway, left] {$a$} (1);
\draw[->-] (3) to [out =60, in=300] node[midway, right] {$d$} (1);

\draw[->-] (1) to [out =200, in=70] node[midway, above] {$a$} (4);
\draw[->-] (4) to [out =110, in=160] node[midway, above left] {$b$} (1);

\draw[->-] (1) to [out =340, in=110] node[midway, above] {$d$} (2);
\draw[->-] (2) to [out =70, in=20] node[midway, above right] {$c$} (1);

\draw[blue,->-] (1) 
        .. controls (8,5) and (4,5) .. (1) 
        node[midway, above, black] {$a$};

\node[above] at (1) {$v$};
\node [right] at(2) {$u$};

\draw[blue, transform canvas={xshift=0.5pt,yshift=0.5pt}](1) -- (2);
\draw[blue, transform canvas={xshift=0.5pt,yshift=-0.5pt}](2) -- (3);
\draw[blue, transform canvas={xshift=-0.5pt,yshift=-0.5pt}](3) -- (4);
\draw[blue, transform canvas={xshift=-0.5pt,yshift=0.5pt}](4) -- (1);

\filldraw (1) circle (2pt);
\filldraw (2) circle (2pt);
\filldraw (3) circle (2pt);
\filldraw (4) circle (2pt);

\end{tikzpicture}

\caption{\label{fig:Gamma after quotient}This is the $\protect\samgamsig$
corresponding to Figure \ref{fig:Gamma before quotient}. The blue
edges indicate the edges that are formed by gluing edges from the
$w$--cycle to edges in the $R$--cycles (or to other edges in the
$w$--cycle). In the language of \S\ref{subsec:Bounding}, these
are the boundary edges of $X$, there is exactly one \emph{piece},
exactly one \emph{external boundary vertex }(labeled $v$) and two
faces, one with boundary $R=R_{2}$ and one with boundary $R^{2}$.}

\end{figure}
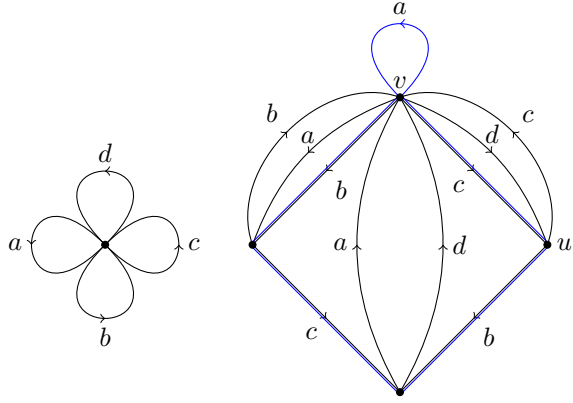

\section{\label{sec:Bounding-Euler-characteristic}Bounding Euler characteristic}

We will prove Proposition \ref{prop:bounding euler char of Gamma -k}
by constructing a combinatorial $2$--complex $X=X\left(\sigma_{x},\pi_{i}\right),$
whose $1$--skeleton, $X^{(1)},$ is exactly $\samgamsig$. We will
then use angle structures to bound the Euler characteristic of $X$
and thus bound $\chi\left(\samgamsig\right).$ We will refer to each
$(\sigma_{x},\pi_{i})\in\mathrm{Match}_{S_{k}}(w,k)$ as a \emph{gluing}
of the $R$--cycles and $w$--cycle, and the process of identifying
vertices and folding edges as gluing as well. To construct $X,$ we
first glue according to the $\pi_{i}$ identifications. The result
is a number, say $m$, disjoint cycles reading out $R_{g}^{k_{i}}$
when traversed, say, clockwise, with $\sum_{i}k_{i}=k$, and the $w$--cycle.
We refer to the disjoint cycles reading each $R_{g}^{k_{i}}$ also
as the $R$--cycles. To construct $\samgamsig$ from these cycles
we simply glue and fold according to the $\sigma_{x}$ identifications.
To construct $X,$ we glue in a disc to each of the $R_{g}^{k_{i}}$--cycles
before gluing according to the $\sigma_{x}$. We refer to these discs
as the faces of $X$. Obviously, we have 
\[
\chi\left(\samgamsig\right)=\chi(X)-m.
\]

\begin{rem}
In principle, one could estimate $\chi\left(\samgamsig\right)$ directly
using the combinatorial formula 
\[
\chi\left(\samgamsig\right)=\sum_{v}1-\frac{\deg(v)}{2}
\]
and keeping track of the terms that appear. We choose to construct
the $2$--complex $X$ and use angle structures (see \S\ref{subsec:Angle-structures})
as they exhibit the sort of cancellations required to obtain Propositions
\ref{prop:bounding euler char of Gamma -k} more easily, as well as
highlighting where the conditions on $w$ come in to effect, discussed
in the following section, but we stress that this construction is
not wholly necessary. 
\end{rem}

The methods of this paper are directly inspired by the approach of
Magee in \cite{Magee2021}. We now to highlight the main difference
between the two papers. In \cite{Magee2021}, Magee uses each matching
datum $(\sigma_{x},\pi_{i})$ to construct a surface whose dual graph
is exactly our $\samgamsig$. The primary difference being, each $\sigma_{x}$
is a permutation as opposed to a set partition. That is, each $x$--labeled
edge is formed by gluing together exactly \emph{two} $x$--labeled
edges. This is not the case in the current paper, since the $\sigma_{x}$--arcs
can define a gluing of many $x$--labeled edges in the $w$--cycle
in to a single edge. Again we emphasise that, although not wholly
necessary, our use of angle structures to bound $\chi\left(\samgamsig\right)$
allows one to overcome these differences more easily. 

\subsection{\label{subsec:Combinatorial-characterisation-o}Combinatorial characterisation
of shortest element in conjugacy class}

We digress briefly to discuss a combinatorial characterisation of
a word $w$ in the alphabet $\{a_{1},b_{1},\dots,a_{g},b_{g}\}$ (i.e.
$w\in F_{2g}$) being a shortest representative of the conjugacy class
of $\gamma\in\Gamma_{g}.$ This idea goes back to the work of Dehn
\cite{Dehn} on the conjugacy problem in surface groups, later refinements
of Birman and Series \cite{BirmanSeries} and an explicit inequality
derived from their work by Magee and Puder \cite{MageePuder2023}.
One can read a similar overview to the one given here in \cite{Magee2021}.

The universal cover $U$ of $\Sigma_{g}$ can be viewed as a disc
tiled by $4g$--gons. The edges are directed and have labels in the
alphabet $\alphabet$, with the boundary of each $4g$--gon, when
read say clockwise, read out the word $R_{g}=[a_{1},b_{1}]\dots[a_{g},b_{g}].$
There is a natural action of $\Sigma_{g}$ on $U$ by translations
obtained by fixing a base point $u$. Given $\gamma\in\Gamma_{g}$
with $\gamma\neq e,$ the quotient 
\[
U/\langle\gamma\rangle
\]
is an annulus $A_{\gamma}$ tiled by infinitely many $4g$--gons,
inheriting the edge labels and directions from $U$. The base point
$u$ maps to a vertex we label $v_{0}.$ 

Pick $w\in F_{2g}$ to be a shortest reduced word representing the
conjugacy class of $\gamma$. Beginning at $v_{0},$ we follow a loop
$\mathcal{L}_{w}$ by following the edges as spelled out by $w$.
Since $w$ is a shortest element representing the conjugacy class
of $\gamma$ and $\gamma\neq e,$ this is a non null--homotopic loop
in the $1$--skeleton of $A_{\gamma},$ $A_{\gamma}^{(1)}.$ Every
vertex in $A_{\gamma}^{(1)}$ has $4g$ incident half--edges with
a well--defined cyclic order on the labels as shown in Figure \ref{fig: star of interior and internal boundary vertex}.
To be precise, this order is 
\[
a_{i}\ \mathrm{out},\ b_{i}\ \mathrm{in},\ a_{i}\ \mathrm{in},\ b_{i}\ \mathrm{out},
\]
repeating cyclically for $i=1,\dots,g$. The loop $\mathcal{L}_{w}$
cuts $A_{\gamma}$ in to two annuli, say $A_{\gamma}^{+}$ and $A_{\gamma}^{-}$,
and each half--edge incident at the vertices in the loop belongs
to exactly one of these.

A piece $P$ is a contiguous collection of edges of $\mathcal{L}_{w}$
defined by $w$ together with the hanging half--edges, at each vertex
taking the hanging half--edges either all in $A_{\gamma}^{+}$ or
all in $A_{\gamma}^{-}$. We write $\mathfrak{e}(P)$ for the number
of loop edges in $P$ and $\mathfrak{he}(P)$ for the number of hanging
half--edges. Birman and Series \cite{BirmanSeries} proved that there
are combinatorial restrictions on the pieces of the loop that can
appear. The following inequality is derived by Magee and Puder as
an almost direct consequence of the work in \cite{BirmanSeries}.
The proof is based on the idea that, if $w$ follows the boundary
of a $4g$--gon for `too long', then it can be shortened using the
relation (i.e. by following the other sides of the $4g$--gon). This
lemma will be crucial in \S\ref{subsec:Bounding}, where we use it
to determine our bound on $\chi\left(\samgamsig\right).$
\begin{lem}[{\cites[Lemma 5.18]{MageePuder2023}[Lemma 4.10]{Magee2021}}]
 \label{lem: bound for e(P), he(P)}Suppose $\gamma\neq e$ and $w\in F_{2g}$
is a shortest element representing the conjugacy class of $\gamma.$
If $P$ is a piece following the entire loop defined by $w$ together
with its hanging half--edges, then 
\[
\mathfrak{e}(P)\le(2g-1)\mathfrak{he}(P).
\]
Otherwise, $P$ follows some subpath of the loop and we have 
\[
\mathfrak{e}(P)\le(2g-1)\mathfrak{he}(P)+2g.
\]
 
\end{lem}

\subsection{\label{subsec:Angle-structures}Angle structures}

Angle structures as introduced by Wise \cite{WiseAngles} offer a
framework for computing the Euler characteristic of a $2$--complex
via the combinatorial Gauss--Bonnet. One assigns an angle to each
corner of a face, giving the associated complex a notion of curvature
through which one can encode the Euler characteristic. For each face
$f,$ we write $\mathcal{C}(f)$ for the collection of corners defined
by two edges in the boundary of $f$ meeting and we write $\angle c$
for the associated angle. We have the following notion of curvature
of a face $f$, 
\[
\kappa(f)\overset{\mathrm{def}}{=}2\pi+\sum_{c\in\mathcal{C}(f)}\left(\angle c-\pi\right).
\]
For each vertex $v$, the link $\mathrm{Link}(v)$ is a graph that
encodes the local structure of the $2$--complex at $v$. It has
a vertex for every edge incident at $v$ and an edge between say $u_{e}$
and $u_{e'}$ if there is a corner $c$ of a face $f$ corresponding
to a meeting of $e$ and $e'$. We write $\mathcal{C}(v)$ for the
collection of all corners $c$ of faces $f$ that are defined by two
edges meeting at $v$ and write $\angle c$ again for the corresponding
angles. There is a notion of curvature at a vertex $v,$
\[
\kappa(v)\overset{\mathrm{def}}{=}2\pi-\pi\chi\left(\link\right)-\sum_{c\in\mathcal{C}(v)}\angle c.
\]
Then we have the following proposition from e.g. \cite[Theorem 2.3]{WiseAngles},
which gives the connection between Euler characteristic of a combinatorial
$2$--complex and the curvature definitions above. 
\begin{prop}[Combinatorial Gauss--Bonnet]
\label{prop:combinatorial gauss-bonnet}For any $2$ complex $X$
with an assigned angle structure, 
\[
2\pi\chi(X)=\sum_{f}\kappa(f)+\sum_{v}\kappa(v).
\]
\end{prop}

\subsection{\label{subsec:Bounding}Proof of Proposition \ref{prop:bounding euler char of Gamma -k}}

From now, we fix any $(\sigma_{x},\pi_{i})\in\mathrm{Match}_{S_{k}}\left(w,k\right)$
arbitrarily. We assign to every corner of a face an angle of $\frac{\pi}{2g}.$
This choice of angle structure is consistent with a hyperbolic metric
of constant negative curvature on a closed genus $g$ surface. Support
for this choice of angle structure is below, in (\ref{eq: interior vertices curvature =00005Cle0}),
whereby it guarantees that we only need to bound vertex curvature
at `boundary vertices', at which stage we can employ Lemma \ref{lem: bound for e(P), he(P)}
to obtain a suitable upper bound. 

Now, suppose that $X=X\left(\sigma_{x},\pi_{i}\right)$ has $m$ faces.
Then, if a face $f$ has boundary reading out $R_{g}^{k_{i}},$ we
have 
\[
\kappa(f)=2\pi+4g\left(\frac{1-2g}{2g}\right)\pi k_{i},
\]
so that 
\begin{equation}
\sum_{f}\kappa(f)=2\pi m-\sum_{f}2(2g-1)\pi k_{i}=2\pi m-2(2g-1)\pi k.\label{eq: curvature of faces}
\end{equation}
Proposition \ref{prop:bounding euler char of Gamma -k} will follow
from the following lemma. 
\begin{lem}
\label{lem:curvature of vertices}With the assigned angle structure
on $X,$ 
\[
\sum_{v}\kappa(v)\le2(2g-2)\pi k.
\]
\end{lem}

Assuming Lemma \ref{lem:curvature of vertices}, we give the following
proof of Proposition \ref{prop:bounding euler char of Gamma -k}.
\begin{proof}[Proof of Proposition \ref{prop:bounding euler char of Gamma -k}]
 Combining Lemma \ref{lem:curvature of vertices} with (\ref{eq: curvature of faces})
and Proposition \ref{prop:combinatorial gauss-bonnet} yields that
\[
2\pi\chi(X)\le2\pi m-2(2g-1)\pi k+2(2g-2)\pi k.
\]
This gives 
\[
\chi(X)\le m-k
\]
and thus 
\[
\chi\left(\samgamsig\right)\le-k
\]
as required. 
\end{proof}
It remains only to prove Lemma \ref{lem:curvature of vertices}. We
will do so by partitioning the vertices in to \emph{interior }and
\emph{boundary }vertices. We first describe the structure of $X^{(1)}$
and in particular, how Lemma \ref{lem: bound for e(P), he(P)} must
be used to bound the curvature contribution from the boundary vertices.
The edges of $X^{(1)}$ are defined by the constituent edges in the
$R$--cycles (we call these $R$--edges) and the $w$--cycle (we
call these $w$--edges) that have been glued together to make up
the edge. We label these by types (and agreeing with the notation
in \cite{Magee2021} wherever possible): 
\begin{itemize}
\item \textbf{WR}: edges that are formed by gluing a single $R$--edge
to at least one $w$--edge.
\item \textbf{RR}: edges that are formed by gluing exactly two $R$--edges
together and no more. 
\item \textbf{WW}: edges that are formed by gluing only $w$--edges together
(possibly just a single $w$--edge).
\item \textbf{RRW}: edges that are formed by gluing two $R$--edges together
and at least one $w$--edge. 
\end{itemize}
Writing $E_{\mathrm{WR}}$ for the collection of WR--edges (and similarly
for the other types), we have the following equality
\begin{equation}
|E_{\mathrm{WR}}|+2|E_{\mathrm{RR}}|+2|E_{\mathrm{RRW}}|=4gk,\label{eq: sum of WR, RR, RRW edges}
\end{equation}
since there are $4gk$ edges in the $R$--cycles. 

The \emph{boundary }of $X$, denoted $\partial X$, is the collection
of paths covered by WR--edges, RRW--edges and WW--edges, together
with the vertices they are incident to. We call the vertices with
boundary edges incident \emph{boundary vertices} and all other vertices
are called \emph{interior vertices}. Every interior vertex $v$ has
at least $4g$ incident edges, following the same pattern as in Figure
\ref{fig: star of interior and internal boundary vertex}, possibly
with the entire pattern repeating some number of times. Thus, for
an interior vertex $v$, $\link$ is a cycle so $\chi\left(\link\right)=0$
and $|\mathcal{C}(v)|\ge4g.$ Hence for any interior vertex $v$,
\begin{equation}
\kappa(v)=2\pi-\frac{\pi}{2g}|\mathcal{C}(v)|\le0.\label{eq: interior vertices curvature =00005Cle0}
\end{equation}
We now deal with the boundary vertices. An \emph{internal }boundary
vertex is one that has exactly two WR--edges or RRW--edges incident
and no WW--edges. All other boundary vertices are called \emph{external}.
If an external boundary vertex $v$ has only WW--edges incident,
then $\mathcal{C}(v)=\emptyset.$ In particular, $\link$ is simply
a disjoint collection of vertices, one for each WW--edge incident
at $v$. Since $w$ is cyclically reduced, there are at least $2$
such edges and so $\chi\left(\link\right)\ge2.$ Hence, for these
vertices, 
\[
\kappa(v)=2\pi-\pi\chi\left(\link\right)\le0
\]
and they can be safely ignored. Thus, to bound 
\[
\sum_{v}\kappa(v),
\]
 we now only need to consider the curvature at the remaining external
boundary vertices and the internal boundary vertices. These are analogous
to the piece--adjacent junction discs and pre--piece discs respectively
in \cite[Section 4.5]{Magee2021}

A \emph{piece $P$ }of $\partial X$ is a subpath of the boundary
consisting of WR--edges or RRW--edges and the \emph{internal} boundary
vertices they are incident to. One should think of a piece as a (collection
of disjoint) subpath(s) of the $w$--cycle that have been glued to
$R$--edges in such a way that it does not intersect with any other
vertices of the $w$--cycle. We distinguish between pieces $P$ consisting
of WR--edges and those consisting of RRW--edges. In particular,
if a piece $P$ consists of RRW--edges, then every internal vertex
$v$ has $\chi\left(\link\right)\ge0$ and $|\mathcal{C}(v)|\ge4g,$
and thus also has 
\[
\kappa(v)\le0.
\]
\emph{Hence, we will ignore the contribution to the curvature from
the internal boundary vertices of pieces consisting of only RRW--edges.} 

Now consider pieces consisting of only WR--edges. In analogy with
\S\ref{subsec:Combinatorial-characterisation-o}, we write $\mathfrak{e}(P)$
for the number of WR--edges in $P$ and $\mathfrak{he}(P)$ for the
number of incident RR--edges at the vertices of $P$, which we refer
to as hanging half--edges. To employ Lemma \ref{lem: bound for e(P), he(P)},
for each piece $P$ we define a piece $\tilde{P}$ of $\mathcal{L}_{w}$
that follows the same collection of edges in $\mathcal{L}_{w}$ as
$P$ does in the $w$--loop. If the WR--edges in a piece $P$ each
consist of at least two $w$--edges, then we simply need to pick
one of the corresponding subpaths of the $w$--cycle to define the
edges of $\tilde{P}.$ The collection of half--edges that we take
depends on the orientation of the gluing. To be precise, we view the
$R$--cycles and the $w$--cycle as reading out $R^{k_{i}}$ and
$w$ when traversed clockwise from some fixed base point, as in Figure
\ref{fig:Gamma before quotient}. Consider first a piece $P$ consisting
of WR--edges, each formed by gluing exactly one $R$--edge to \emph{exactly
one} $w$--edge. We always view such a piece $P$ as corresponding
to some subpath of the $w$--cycle that is traversed following the
clockwise direction. We say $P$ has been glued clockwise if each
corresponding edge in the $w$--loop is glued to an edge with the
same label and direction in the $R$--cycles \emph{when traversed
clockwise} (e.g. an $a$--labeled edge in the $w$--loop is glued
to an $a$--labeled edge that is also directed clockwise in the $R$--cycles,
corresponding to $a$ in $R$ rather than $a^{-1}$). In this case,
for $\tilde{P},$ we take all hanging edges on the \emph{left} as
$\mathcal{L}_{w}$ is traversed in its assigned direction (i.e. corresponding
to reading $w$ along $\mathcal{L}_{w}$). Conversely, we say $P$
has been glued anticlockwise if each corresponding edge in the $w$--loop
is glued to an edge in the $R$--cycles that is directed anticlockwise,
and in such a case for $\tilde{P}$ we take all hanging edges on the
\emph{right }as we read $\mathcal{L}_{w}$ in reverse. 

Clearly, in both cases we have 
\[
\mathfrak{e}(P)=\mathfrak{e}\left(\tilde{P}\right)
\]
 and 
\[
\mathfrak{he}\left(\tilde{P}\right)\le\mathfrak{he}(P).
\]
The inequality here reflects that, at each internal boundary vertex,
the entire pattern of incident half--edges may repeat some number
of times. It follows from Lemma \ref{lem: bound for e(P), he(P)}
that, if $P$ consists of the whole $w$--loop (i.e. it is the only
piece of $\partial X$ and it is a cycle and all boundary vertices
are internal) then 
\[
\mathfrak{e}(P)\le(2g-1)\mathfrak{he}(P)
\]
and otherwise 
\[
\mathfrak{e}(P)\le(2g-1)\mathfrak{he}(P)+2g.
\]

\begin{example}
One can determine the hanging half edges incident at each internal
boundary vertex by considering Figure \ref{fig: star of interior and internal boundary vertex}.
In Figure \ref{fig:Gamma after quotient}, there is exactly one piece
$P$ corresponding the subpath reading $cbc^{-1}b^{-1}$ of the $w$--cycle
in Figure \ref{fig:Gamma before quotient}. This piece has been glued
anticlockwise. At the first internal vertex, labeled $u$, we have
`$c$ in' followed by `$b$ out' as we read $w$ cyclically clockwise.
Considering Figure \ref{fig: star of interior and internal boundary vertex}
(here, $c=a_{2}$ and $b=b_{1}$), we take the hanging edges to the
\emph{right} of these edges as we follow $c$ in and then $b$ out.These
are `$b_{2}$ in' and `$a_{2}$ out' (here, $b_{2}=d$ and $a_{2}=c$),
and we see in Figure \ref{fig:Gamma after quotient} that these are
exactly the emanating edges from this internal boundary vertex. 
\end{example}

In the case whereby there are multiple $w$--edges glued together
to form each WR--edge, we simply pick one of the corresponding subpaths
of the $w$--cycle. If the subpath we choose spells out the subword
$u$ of $w$ when traversed say clockwise, then the other constituent
subpaths of the piece that are glued with the same orientation must
also read $u$, and the subpaths glued with the opposite orientation
must read $u^{-1}$ when traversed clockwise. 

\begin{defn}
Given a piece $P$, a subpiece of $P$ is a chain of WR--edges that
are contained within the chain of WR--edges that define $P$, together
with the internal boundary vertices they are incident to (excluding
the endpoints) and the hanging half--edges at these vertices. We
denote the collection of subpieces by $\mathrm{sub}(P).$ 
\end{defn}

As above, any $p\in\mathrm{sub}(P)$ defines a piece $\tilde{p}$
which must satisfy Lemma \ref{lem: bound for e(P), he(P)}. In particular,
if $p$ is not a full cycle, then it must satisfy $\mathfrak{e}(p)\le(2g-1)\mathfrak{he}(p)+2g$
and otherwise it must satisfy $\mathfrak{e}(p)\le(2g-1)\mathfrak{he}(p)$.
For an internal boundary vertex $v$, we denote by $\mathfrak{he}(v)$
the number of hanging half--edges at $v.$ We have 
\[
|\mathcal{C}(v)|=\mathfrak{he}(v)+1
\]
and 
\[
\mathfrak{he}(P)=\sum_{v\in P}\mathfrak{he}(v).
\]
For such a vertex, $\link$ has a very specific form related to the
incident WR--edges and hanging half--edges, see Figure \ref{fig: star of interior and internal boundary vertex}.
Indeed, it is always an interval and thus has $\chi=1$. Hence, 
\begin{equation}
\kappa(v)=2\pi-\pi\chi\left(\link\right)-\frac{\pi}{2g}|\mathcal{C}(v)|\le\pi\left(1-\frac{\mathfrak{he}(v)+1}{2g}\right).\label{eq: curvature of internal boundary vertex}
\end{equation}

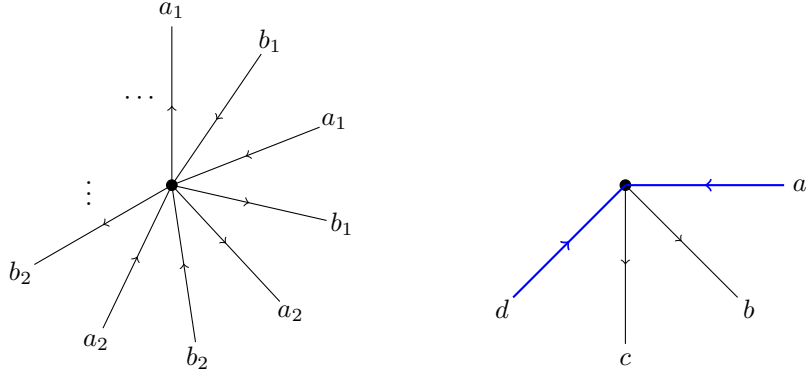
\begin{figure}[H]
\centering\begin{tikzpicture}[scale=1, baseline=(current bounding box.center)]

\def\r{3}
\coordinate (O) at (0,0); \filldraw (O) circle(2pt);

\pgfmathsetmacro{\step}{240/7}

\foreach \k in {0,...,7}{

    \pgfmathsetmacro{\i}{90 - \k*\step}
    \pgfmathsetmacro{\s}{0.7}

    \coordinate (P\k) at ({\s*\r*cos(\i)},{\s*\r*sin(\i)});
}

    \draw[->-] (O)--(P0);
\draw[->-] (O)--(P3);
\draw[->-] (O)--(P4);
\draw[->-] (O)--(P7);
\draw[->-] (P1)--(O);
\draw[->-] (P2)--(O);
\draw[->-] (P5)--(O);
\draw[->-] (P6)--(O);

\node at ($(O)!1.1!(P0)$) {$a_1$};
\node at ($(O)!1.1!(P1)$) {$b_1$};
\node at ($(O)!1.1!(P2)$) {$a_1$};
\node at ($(O)!1.1!(P3)$) {$b_1$};
\node at ($(O)!1.1!(P4)$) {$a_2$};
\node at ($(O)!1.1!(P5)$) {$b_2$};
\node at ($(O)!1.1!(P6)$) {$a_2$};
\node at ($(O)!1.1!(P7)$) {$b_2$};

\coordinate (u) at (-1.1, 0); \node at (u) {$\vdots$};
\coordinate (v) at (-0.4, 1.15); \node at (v){$\cdots$};

\coordinate (U) at (6,0); \filldraw (U) circle(2pt);

\foreach \k in {0,...,7}{
\pgfmathsetmacro{\s}{0.7}

    \pgfmathsetmacro{\i}{90 - 45*\k}

    \coordinate (C\k) at
        ($(U)+({\s*\r*cos(\i)},{\s*\r*sin(\i)})$);
}

\draw[blue, thick,  ->-] (C2)--(U);
\draw[->-] (U) -- (C3); 
\draw[->-] (U) -- (C4); 
\draw[blue, thick, ->-] (C5)--(U);

\node at ($(U)!1.1!(C2)$) {$a$};
\node at ($(U)!1.1!(C3)$) {$b$};
\node at ($(U)!1.1!(C4)$) {$c$};
\node at ($(U)!1.1!(C5)$) {$d$};

\end{tikzpicture}

\caption{\label{fig: star of interior and internal boundary vertex}On the
left is the star of an interior vertex in $X^{(1)}=\protect\samgamsig$.
It has $4g$ incident edges and the labels of these incident edges
are exactly the same as the labels of the incident half--edges for
the vertices in $A_{\gamma}$ as in \S\ref{subsec:Combinatorial-characterisation-o}.
Clearly, the link of this vertex is a cycle with $4g$ edges in total,
since there is a face between each pair of consecutive edges.\protect \\
On the right is the star of an internal boundary vertex $v$ (in some
piece $P$) in $X^{(1)},$ for $g=2$. The blue edges are WR--edges,
corresponding to some subword(s) `$ad^{-1}$' of $w$ glued in the
clockwise direction (or to $da^{-1}$, glued anticlockwise). Evidently
$\mathfrak{he}(v)=2$ and thus $|\mathcal{C}(v)|=3$. These half--edges
correspond to the half--edges incident at each corresponding vertex
in $\hat{P}$ in $\mathcal{L}_{w}$, containing only those on the
left when traversed in the assigned orientation (respectively, only
those on the right if glued anticlockwise and traversing $\mathcal{L}_{w}$
in reverse. Clearly, the link of this vertex is a path with $4$ vertices
and $3$ edges, and thus has $\chi=1.$ }

\end{figure}

\subsubsection{\label{subsec:Proof-of-Lemma INTERNAL}Proof of Lemma \ref{lem:curvature of vertices}
when all boundary vertices are internal}

Firstly, the lemma is obvious if the unique piece $P$ consists of
only RRW--edges, since then $\sum_{v}\kappa(v)\le0.$ So assume it
only consists of WR--edges. Denote the set of internal (and hence
all) boundary vertices by $V_{w}.$ From the previous discussion,
interior vertices give a non--positive contribution to curvature
so can be safely dispensed with. Thus, using (\ref{eq: curvature of internal boundary vertex}),
we have 
\begin{equation}
\sum_{v}\kappa(v)\le\pi\left(|V_{w}|-\sum_{v\in V_{w}}\frac{\mathfrak{he}(v)+1}{2g}\right)=\pi\left(\frac{2g-1}{2g}|V_{w}|-\sum_{v\in V_{w}}\frac{\mathfrak{he}(v)}{2g}\right).\label{eq: curvature in terms of V-=00005Cgamma}
\end{equation}
The unique piece $P$ of $\partial X$ is the entire boundary cycle.
Using Lemma \ref{lem: bound for e(P), he(P)}, we have that 
\[
\begin{aligned}|V_{w}|=\mathfrak{e}(P) & \le(2g-1)\mathfrak{he}(P)\\
 & =(2g-1)\left(\sum_{v\in V_{w}}\mathfrak{he}(v)\right)
\end{aligned}
\]
and so
\[
\frac{1}{2g(2g-1)}|V_{w}|\le\sum_{v}\frac{\mathfrak{he}(v)}{2g}.
\]
Combining this with (\ref{eq: curvature in terms of V-=00005Cgamma})
gives 
\begin{equation}
\sum_{v}\kappa(v)\le\pi\left(\frac{2g-2}{2g-1}\right)|V_{w}|.\label{eq: sum of curvature of boundary, INTERNAL ONLY}
\end{equation}
Using (\ref{eq: sum of WR, RR, RRW edges}), we have $|E_{\mathrm{WR}}|+2|E_{\mathrm{RR}}|\le4gk$.
Since hanging half--edges are RR--edges, and each such edge has
at most two endpoints that could contribute to $\mathfrak{he}(P)$,
we have $\mathfrak{he}(P)\le2|E_{\mathrm{RR}}|.$ We also have $|E_{\mathrm{WR}}|=|V_{w}|,$
giving 
\[
|V_{w}|+\mathfrak{he}(P)\le4gk.
\]
Lemma \ref{lem: bound for e(P), he(P)} gives $\frac{1}{2g-1}|V_{w}|\le\mathfrak{he}(P)$
and thus 
\[
\left(\frac{2g}{2g-1}\right)|V_{w}|\le4gk,
\]
which gives 
\[
|V_{w}|\le2(2g-1)k.
\]
Combining this with (\ref{eq: sum of curvature of boundary, INTERNAL ONLY})
proves Lemma \ref{lem:curvature of vertices}.

\subsubsection{\label{subsec:Proof-of-Lemma DEGENERATE}Proof of Lemma \ref{lem:curvature of vertices}
with degenerate boundary}

We now deal with the case of a degenerate boundary i.e. where the
boundary is not a single cycle. In this case we must employ the second
inequality in Lemma \ref{lem: bound for e(P), he(P)} for the internal
boundary vertices and deal with the external boundary vertices separately.
We denote these sets by $V_{w}^{\mathrm{int}}$ and $V_{w}^{\mathrm{ext}}$
respectively and we split our sum 
\[
\sum_{v\in V_{w}}\kappa(v)=\sum_{v\in V_{w}^{\mathrm{int}}}\kappa(v)+\sum_{v\in V_{w}^{\mathrm{ext}}}\kappa(v).
\]
We first describe a simplifying procedure, designed to deal with external
boundary vertices with higher curvature than we want. This procedure,
although having a slightly technical formal description, is a simple
idea designed to ensure we are efficient with our use of Lemma \ref{lem: bound for e(P), he(P)}.
\begin{rem}
These `high curvature' vertices are present in the current paper and
are not present in \cite{Magee2021}. They arise from having pieces
$P$ that a single RRW--edge which can force an external boundary
vertex $v$ to have $\chi\left(\link\right)=1$, which is not possible
without such edges. 
\end{rem}

\subsubsection*{Piece unzipping procedure}

As mentioned previously, here we deal with external boundary vertices
$v$ with $\chi\left(\link\right)=1$ and $\kappa(v)=\frac{-\pi}{2g}.$
These are vertices with exactly three WR or RRW--edges incident.
In fact, they must have exactly two WR--edges incident, labeled $f_{1}$
and $f_{2}$, and one RRW--edge incident whereby the RRW--edge is
formed by gluing two $R$--edges to at least two $w$--edges, as
in Figure \ref{fig: unzipping}. We label the pieces that $f_{1}$
and $f_{2}$ belong to by $P_{1}$ and $P_{2}$ and the piece containing
the RRW--edge by $P$ and the vertex by $v_{0}.$ We label the RRW--edge
by $e_{0}.$ Consider the following procedure performed on $X^{(1)}=\samgamsig,$
which we refer to as `unzipping' the RRW--edge $e_{0}.$ We split
$v_{0}$ in to two vertices, $v_{0}'$ and $v_{0}''$ and split $e_{0}$
in to two edges, $e_{0}'$ and $e_{0}''$. Equivalently, we remove
the arcs between the two $R$--edges that were glued together to
form $e_{0}$, as well as the arcs between the edges in the $w$--cycle
that are adjacent to the $w$--cycle edges that make up $f_{1}$
and $f_{2}$ respectively. Now, $v_{0}'$ and $v_{0}''$ are considered
internal boundary vertices, with $P_{1}$ now extended to include
$v_{0}'$ and $e_{0}'$ and $P_{2}$ now extended to include $v_{0}''$
and $e_{0}''$. 

There are now two distinct possibilities -- either $v_{1}$ was an
\emph{external }boundary vertex or it was an \emph{internal} boundary
vertex.. Firstly, if the endpoint $v_{1}$ of $e_{0}$ was an \emph{external
}boundary vertex with $\chi\left(\mathrm{Link}(v_{1})\right)=0,$
after unzipping it now has $\chi\left(\mathrm{Link}(v_{1})\right)=1$
and necessarily has $|\mathcal{C}(v)|\ge4g$, and hence it now has
$\kappa(v_{1})\le-\pi.$ Secondly, if $v_{1}$ was an external boundary
vertex with $\chi\left(\mathrm{Link}(v_{1})\right)=1$, then unzipping
either splits $v_{1}$ into two distinct vertices, in which case $P_{1}$
and $P_{2}$ extend through these vertices (i.e. they become internal
vertices) \emph{or }unzipping forces $v_{1}$ to have $\chi\left(\mathrm{Link}(v_{1})\right)\ge2$.
Finally, if $v_{1}$ was an an external boundary vertex with $\chi\left(\mathrm{Link}(v_{1})\right)\ge2,$
we are done after unzipping. An example of the first two cases can
be seen in Figure \ref{fig: unzipping}. It is clear this unzipping
procedure can only increase $\chi\left(\samgamsig\right).$ Finally,
if the endpoint $v_{1}$ of $e_{0}$ was an \emph{internal }boundary
vertex then we repeat the procedure, `unzipping' the RRW--edge, say
$e_{1}$, emanating from $v_{1}$. We iterate $j$ times until we
reach an external boundary vertex. 

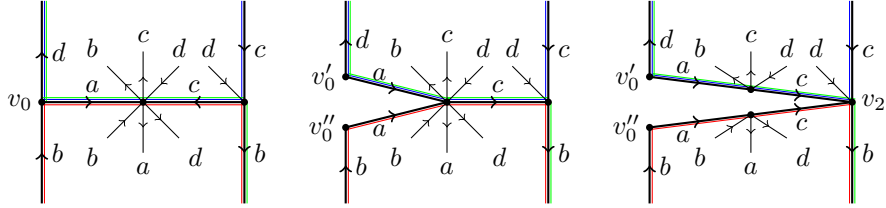
\begin{figure}[H]
\centering\begin{tikzpicture}[scale=0.67, baseline=(current bounding box.center)]

\coordinate (A) at (0,0);
\coordinate (B) at (2,0);
\coordinate (C) at (4,0);
\coordinate (Au) at (0,2); \coordinate (Ad) at (0,-2);
\coordinate (Cu) at (4,2); \coordinate (Cd) at (4,-2);
\draw[red, transform canvas={xshift=1pt,yshift=0pt}] (Ad) -- (A);
\draw[red, transform canvas={xshift=0pt,yshift=-1pt}] (A) -- (B);
\draw[red, transform canvas={xshift=0pt,yshift=-1pt}] (B) -- (C);
\draw[red, transform canvas={xshift=-1pt,yshift=0pt}] (C) -- (Cd);
\draw[blue, transform canvas={xshift=1pt,yshift=0pt}](Au) -- (A); 
\draw[blue, transform canvas={xshift=0pt, yshift=1pt}] (A)--(B);
\draw[blue, transform canvas={xshift=0pt, yshift=1pt}] (B)--(C);
\draw[blue, transform canvas={xshift=-1pt, yshift=0pt}] (C)--(Cu);
\draw[thick, ->-] (A) -- node[midway, right]{$d$}(Au); 
\draw[thick, ->-] (Ad) -- node[midway, right]{$b$}(A);
\draw[thick, ->-] (A) -- node[midway, above]{$a$}(B);
\draw[thick, ->-] (C) -- node[midway, above]{$c$}(B);
\draw[thick, ->-] (C) -- node[midway, right]{$b$}(Cd);
\draw[thick, ->-] (Cu) -- node[midway, right]{$c$}(C);
\fill (A) circle (2pt);
\fill (B) circle (2pt);
\fill (C) circle (2pt);
\coordinate (he1) at (1.29, 0.71); \coordinate (he2) at (2, 1); \coordinate (he3) at (2.71, 0.71); 
\coordinate (he4) at (1.29, -0.71); \coordinate (he5) at (2,-1); \coordinate (he6) at (2.71, -0.71);
\coordinate (he7) at (3.29, 0.71);
\draw[->-] (B)--(he1); \draw[->-] (B)--(he2); \draw[->-] (he3)--(B);
\draw[->-] (he4)--(B); \draw[->-] (B)--(he5); \draw[->-] (B)--(he6);
\draw[->-] (he7)--(C);
\node [above left] at (he1) {$b$}; \node [above] at (he2) {$c$}; \node [above] at (he3) {$d$};
\node [below left] at (he4) {$b$}; \node [below] at (he5) {$a$}; \node [below right] at (he6) {$d$};
\node[above] at (he7) {$d$};
\node[left] at (A) {$v_0$};

\coordinate (D1) at (6,0.5); \coordinate (D2) at (6,-0.5);
\coordinate (E) at (8,0);
\coordinate (F) at (10,0);
\coordinate (Du) at (6,2); \coordinate (Dd) at (6,-2);
\coordinate (Fu) at (10,2); \coordinate (Fd) at (10,-2);
\draw[red, transform canvas={xshift=1pt,yshift=0pt}] (Dd) -- (D2);
\draw[red, transform canvas={xshift=0.71pt,yshift=-0.71pt}] (D2) -- (E);
\draw[red, transform canvas={xshift=0pt,yshift=-1pt}] (E) -- (F);
\draw[red, transform canvas={xshift=-1pt,yshift=0pt}] (F) -- (Fd);
\draw[blue, transform canvas={xshift=1pt,yshift=0pt}](Du) -- (D1); 
\draw[blue, transform canvas={xshift=0.71pt, yshift=0.71pt}] (D1)--(E);
\draw[blue, transform canvas={xshift=0pt, yshift=1pt}] (E)--(F);
\draw[blue, transform canvas={xshift=-1pt, yshift=0pt}] (F)--(Fu);
\draw[thick, ->-] (D1) -- node[midway, right]{$d$}(Du); 
\draw[thick, ->-] (Dd) -- node[midway, right]{$b$}(D2);
\draw[thick, ->-] (D1) -- node[midway, above left]{$a$}(E); \draw[thick, ->-] (D2) -- node[midway, below left]{$a$}(E);
\draw[thick, ->-] (E) -- node[midway, above]{$c$}(F);
\draw[thick, ->-] (F) -- node[midway, right]{$b$}(Fd);
\draw[thick, ->-] (Fu) -- node[midway, right]{$c$}(F);
\fill (D1) circle (2pt);\fill (D2) circle (2pt);
\fill (E) circle (2pt);
\fill (F) circle (2pt);
\coordinate (he11) at (7.29, 0.71); \coordinate (he21) at (8, 1); \coordinate (he31) at (8.71, 0.71); 
\coordinate (he41) at (7.29, -0.71); \coordinate (he51) at (8,-1); \coordinate (he61) at (8.71, -0.71);
\coordinate (he71) at (9.29, 0.71);
\draw[->-] (E)--(he11); \draw[->-] (E)--(he21); \draw[->-] (he31)--(E);
\draw[->-] (he41)--(E); \draw[->-] (E)--(he51); \draw[->-] (E)--(he61);
\draw[->-] (he71)--(F);
\node [above left] at (he11) {$b$}; \node [above] at (he21) {$c$}; \node [above] at (he31) {$d$};
\node [below left] at (he41) {$b$}; \node [below] at (he51) {$a$}; \node [below right] at (he61) {$d$};
\node[above] at (he71) {$d$};
\node[left] at (D1) {$v_0'$}; \node[left] at (D2) {$v_0''$};

\coordinate (G1) at (12,0.5); \coordinate (G2) at (12,-0.5);
\coordinate (H1) at (14,0.25); \coordinate (H2) at (14,-0.25);
\coordinate (I) at (16,0);
\coordinate (Gu) at (12,2); \coordinate (Gd) at (12,-2);
\coordinate (Iu) at (16,2); \coordinate (Id) at (16,-2);
\draw[red, transform canvas={xshift=1pt,yshift=0pt}] (Gd) -- (G2);
\draw[red, transform canvas={xshift=0.71pt,yshift=-0.71pt}] (G2) -- (H2);
\draw[red, transform canvas={xshift=0.71pt,yshift=-0.71pt}] (H2) -- (I);
\draw[red, transform canvas={xshift=-1pt,yshift=0pt}] (I) -- (Id);
\draw[blue, transform canvas={xshift=1pt,yshift=0pt}](Gu) -- (G1); 
\draw[blue, transform canvas={xshift=0.71pt, yshift=0.71pt}] (G1)--(H1);
\draw[blue, transform canvas={xshift=0.71pt, yshift=0.71pt}] (H1)--(I);
\draw[blue, transform canvas={xshift=-1pt, yshift=0pt}] (I)--(Iu);
\draw[thick, ->-] (G1) -- node[midway, right]{$d$}(Gu); 
\draw[thick, ->-] (Gd) -- node[midway, right]{$b$}(G2);
\draw[thick, ->-] (G1) -- node[midway, above left]{$a$}(H1); \draw[thick, ->-] (G2) -- node[midway, below left]{$a$}(H2);
\draw[thick, ->-] (H1) -- node[midway, above]{$c$}(I); \draw[thick, ->-] (H2) -- node[midway, below]{$c$}(I);
\draw[thick, ->-] (I) -- node[midway, right]{$b$}(Id);
\draw[thick, ->-] (Iu) -- node[midway, right]{$c$}(I);
\fill (H1) circle (2pt);\fill (H2) circle (2pt);
\fill (G1) circle (2pt); \fill (G2) circle (2pt);
\fill (I) circle (2pt);
\coordinate (he111) at (13.29, 0.71); \coordinate (he211) at (14, 1); \coordinate (he311) at (14.71, 0.71); 
\coordinate (he411) at (13.29, -0.71); \coordinate (he511) at (14,-1); \coordinate (he611) at (14.71, -0.71);
\coordinate (he711) at (15.29, 0.71);
\draw[->-] (H1)--(he111); \draw[->-] (H1)--(he211); \draw[->-] (he311)--(H1);
\draw[->-] (he411)--(H2); \draw[->-] (H2)--(he511); \draw[->-] (H2)--(he611);
\draw[->-] (he711)--(I);
\node [above left] at (he111) {$b$}; \node [above] at (he211) {$c$}; \node [above] at (he311) {$d$};
\node [below left] at (he411) {$b$}; \node [below] at (he511) {$a$}; \node [below right] at (he611) {$d$};
\node[above] at (he711) {$d$};
\node[left] at (G1) {$v_0'$}; \node[left] at (G2) {$v_0''$};
\node[right] at (I) {$v_2$};

\coordinate (X1) at (18,0.5); \coordinate (X2) at (18,-0.5);
\coordinate (Y1) at (20,0.5); \coordinate (Y2) at (20,-0.5);
\coordinate (Z1) at (22,0.5); \coordinate (Z2) at (22,-0.5);
\coordinate (Xu) at (18,2); \coordinate (Xd) at (18,-2);
\coordinate (Zu) at (22,2); \coordinate (Zd) at (22,-2);
\draw[red, transform canvas={xshift=1pt,yshift=0pt}] (Xd) -- (X2);
\draw[red, transform canvas={xshift=0pt,yshift=-1pt}] (X2) -- (Y2);
\draw[red, transform canvas={xshift=0pt,yshift=-1pt}] (Y2) -- (Z2);
\draw[red, transform canvas={xshift=-1pt,yshift=0pt}] (Z2) -- (Zd);
\draw[blue, transform canvas={xshift=1pt,yshift=0pt}](Xu) -- (X1); 
\draw[blue, transform canvas={xshift=0pt, yshift=1pt}] (X1)--(Y1);
\draw[blue, transform canvas={xshift=0pt, yshift=1pt}] (Y1)--(Z1);
\draw[blue, transform canvas={xshift=-1pt, yshift=0pt}] (Z1)--(Zu);
\draw[thick, ->-] (X1) -- node[midway, right]{$d$}(Xu); 
\draw[thick, ->-] (Xd) -- node[midway, right]{$b$}(X2);
\draw[thick, ->-] (X1) -- node[midway, above left]{$a$}(Y1); \draw[thick, ->-] (X2) -- node[midway, below left]{$a$}(Y2);
\draw[thick, ->-] (Y1) -- node[midway, above]{$c$}(Z1); \draw[thick, ->-] (Y2) -- node[midway, below]{$c$}(Z2);
\draw[thick, ->-] (Z2) -- node[midway, right]{$b$}(Zd);
\draw[thick, ->-] (Zu) -- node[midway, right]{$c$}(Z1);
\fill (X1) circle (2pt);\fill (X2) circle (2pt);
\fill (Y1) circle (2pt); \fill (Y2) circle (2pt);
\fill (Z1) circle (2pt); \fill (Z2) circle (2pt);
\coordinate (he1111) at (19.29, 0.91); \coordinate (he2111) at (20, 1); \coordinate (he3111) at (20.71, 0.91); 
\coordinate (he4111) at (19.29, -0.91); \coordinate (he5111) at (20,-1); \coordinate (he6111) at (20.71, -0.91);
\coordinate (he7111) at (21.49, 0.91);
\draw[->-] (Y1)--(he1111); \draw[->-] (Y1)--(he2111); \draw[->-] (he3111)--(Y1);
\draw[->-] (he4111)--(Y2); \draw[->-] (Y2)--(he5111); \draw[->-] (Y2)--(he6111);
\draw[->-] (he7111)--(Z1);
\node [above left] at (he1111) {$b$}; \node [above] at (he2111) {$c$}; \node [above] at (he3111) {$d$};
\node [below left] at (he4111) {$b$}; \node [below] at (he5111) {$a$}; \node [below right] at (he6111) {$d$};
\node[above] at (he7111) {$d$};
\node[left] at (X1) {$v_0'$}; \node[left] at (X2) {$v_0''$};
\node[right] at (Z1) {$v_2'$}; \node[right] at (Z2) {$v_2''$};

\end{tikzpicture}

\caption{\label{fig: unzipping}`Unzipping' a chain of RRW--edges. In both
diagrams, we unzip along the central RRW--edges labeled by $a$ and
$c$. The other edges are RW--edges (except for the hanging half--edges).
The coloured lines are to indicate the distinct subpaths of the $w$--cycle
that have been glued to $R$--edges to form the edges of $X^{(1)}.$
In the above diagram, the final vertex $v_{2}$ splits in to two vertices
as a result of the unzipping. In the below diagram, there is an additional
subpath of the $w$--cycle (in green) which means that, after unzipping
the $c$--labeled edge ($e_{2}$ in the description of the procedure),
the vertex $v_{2}$ does not split in to two vertices. }
\medskip{}

\centering\begin{tikzpicture}[scale=0.67, baseline=(current bounding box.center)]

\coordinate (A) at (0,0);
\coordinate (B) at (2,0);
\coordinate (C) at (4,0);
\coordinate (Au) at (0,2); \coordinate (Ad) at (0,-2);
\coordinate (Cu) at (4,2); \coordinate (Cd) at (4,-2);
\draw[red, transform canvas={xshift=1pt,yshift=0pt}] (Ad) -- (A);
\draw[red, transform canvas={xshift=0pt,yshift=-1pt}] (A) -- (B);
\draw[red, transform canvas={xshift=0pt,yshift=-1pt}] (B) -- (C);
\draw[red, transform canvas={xshift=-1pt,yshift=0pt}] (C) -- (Cd);
\draw[blue, transform canvas={xshift=1pt,yshift=0pt}](Au) -- (A); 
\draw[blue, transform canvas={xshift=0pt, yshift=1pt}] (A)--(B);
\draw[blue, transform canvas={xshift=0pt, yshift=1pt}] (B)--(C);
\draw[blue, transform canvas={xshift=-1pt, yshift=0pt}] (C)--(Cu);
\draw[green, transform canvas={xshift=1.7pt,yshift=0pt}](Au) -- (A); 
\draw[green, transform canvas={xshift=0pt, yshift=1.7pt}] (A)--(B);
\draw[green, transform canvas={xshift=0pt, yshift=1.7pt}] (B)--(C);
\draw[green, transform canvas={xshift=1pt, yshift=0pt}] (C)--(Cd);
\draw[thick, ->-] (A) -- node[midway, right]{$d$}(Au); 
\draw[thick, ->-] (Ad) -- node[midway, right]{$b$}(A);
\draw[thick, ->-] (A) -- node[midway, above]{$a$}(B);
\draw[thick, ->-] (C) -- node[midway, above]{$c$}(B);
\draw[thick, ->-] (C) -- node[midway, right]{$b$}(Cd);
\draw[thick, ->-] (Cu) -- node[midway, right]{$c$}(C);
\fill (A) circle (2pt);
\fill (B) circle (2pt);
\fill (C) circle (2pt);
\coordinate (he1) at (1.29, 0.71); \coordinate (he2) at (2, 1); \coordinate (he3) at (2.71, 0.71); 
\coordinate (he4) at (1.29, -0.71); \coordinate (he5) at (2,-1); \coordinate (he6) at (2.71, -0.71);
\coordinate (he7) at (3.29, 0.71);
\draw[->-] (B)--(he1); \draw[->-] (B)--(he2); \draw[->-] (he3)--(B);
\draw[->-] (he4)--(B); \draw[->-] (B)--(he5); \draw[->-] (B)--(he6);
\draw[->-] (he7)--(C);
\node [above left] at (he1) {$b$}; \node [above] at (he2) {$c$}; \node [above] at (he3) {$d$};
\node [below left] at (he4) {$b$}; \node [below] at (he5) {$a$}; \node [below right] at (he6) {$d$};
\node[above] at (he7) {$d$};
\node[left] at (A) {$v_0$};

\coordinate (D1) at (6,0.5); \coordinate (D2) at (6,-0.5);
\coordinate (E) at (8,0);
\coordinate (F) at (10,0);
\coordinate (Du) at (6,2); \coordinate (Dd) at (6,-2);
\coordinate (Fu) at (10,2); \coordinate (Fd) at (10,-2);
\draw[red, transform canvas={xshift=1pt,yshift=0pt}] (Dd) -- (D2);
\draw[red, transform canvas={xshift=0.71pt,yshift=-0.71pt}] (D2) -- (E);
\draw[red, transform canvas={xshift=0pt,yshift=-1pt}] (E) -- (F);
\draw[red, transform canvas={xshift=-1pt,yshift=0pt}] (F) -- (Fd);
\draw[blue, transform canvas={xshift=1pt,yshift=0pt}](Du) -- (D1); 
\draw[blue, transform canvas={xshift=0.71pt, yshift=0.71pt}] (D1)--(E);
\draw[blue, transform canvas={xshift=0pt, yshift=1pt}] (E)--(F);
\draw[blue, transform canvas={xshift=-1pt, yshift=0pt}] (F)--(Fu);
\draw[green, transform canvas={xshift=1.7pt,yshift=0pt}](Du) -- (D1); 
\draw[green, transform canvas={xshift=1.2pt, yshift=1.2pt}] (D1)--(E);
\draw[green, transform canvas={xshift=0pt, yshift=1.7pt}] (E)--(F);
\draw[green, transform canvas={xshift=1pt, yshift=0pt}] (F)--(Fd);
\draw[thick, ->-] (D1) -- node[midway, right]{$d$}(Du); 
\draw[thick, ->-] (Dd) -- node[midway, right]{$b$}(D2);
\draw[thick, ->-] (D1) -- node[midway, above left]{$a$}(E); \draw[thick, ->-] (D2) -- node[midway, below left]{$a$}(E);
\draw[thick, ->-] (E) -- node[midway, above]{$c$}(F);
\draw[thick, ->-] (F) -- node[midway, right]{$b$}(Fd);
\draw[thick, ->-] (Fu) -- node[midway, right]{$c$}(F);
\fill (D1) circle (2pt);\fill (D2) circle (2pt);
\fill (E) circle (2pt);
\fill (F) circle (2pt);
\coordinate (he11) at (7.29, 0.71); \coordinate (he21) at (8, 1); \coordinate (he31) at (8.71, 0.71); 
\coordinate (he41) at (7.29, -0.71); \coordinate (he51) at (8,-1); \coordinate (he61) at (8.71, -0.71);
\coordinate (he71) at (9.29, 0.71);
\draw[->-] (E)--(he11); \draw[->-] (E)--(he21); \draw[->-] (he31)--(E);
\draw[->-] (he41)--(E); \draw[->-] (E)--(he51); \draw[->-] (E)--(he61);
\draw[->-] (he71)--(F);
\node [above left] at (he11) {$b$}; \node [above] at (he21) {$c$}; \node [above] at (he31) {$d$};
\node [below left] at (he41) {$b$}; \node [below] at (he51) {$a$}; \node [below right] at (he61) {$d$};
\node[above] at (he71) {$d$};
\node[left] at (D1) {$v_0'$}; \node[left] at (D2) {$v_0''$};

\coordinate (G1) at (12,0.5); \coordinate (G2) at (12,-0.5);
\coordinate (H1) at (14,0.25); \coordinate (H2) at (14,-0.25);
\coordinate (I) at (16,0);
\coordinate (Gu) at (12,2); \coordinate (Gd) at (12,-2);
\coordinate (Iu) at (16,2); \coordinate (Id) at (16,-2);
\draw[red, transform canvas={xshift=1pt,yshift=0pt}] (Gd) -- (G2);
\draw[red, transform canvas={xshift=0.71pt,yshift=-0.71pt}] (G2) -- (H2);
\draw[red, transform canvas={xshift=0.71pt,yshift=-0.71pt}] (H2) -- (I);
\draw[red, transform canvas={xshift=-1pt,yshift=0pt}] (I) -- (Id);
\draw[blue, transform canvas={xshift=1pt,yshift=0pt}](Gu) -- (G1); 
\draw[blue, transform canvas={xshift=0.71pt, yshift=0.71pt}] (G1)--(H1);
\draw[blue, transform canvas={xshift=0.71pt, yshift=0.71pt}] (H1)--(I);
\draw[blue, transform canvas={xshift=-1pt, yshift=0pt}] (I)--(Iu);
\draw[green, transform canvas={xshift=1.7pt,yshift=0pt}](Gu) -- (G1); 
\draw[green, transform canvas={xshift=1.2pt, yshift=1.2pt}] (G1)--(H1);
\draw[green, transform canvas={xshift=1.2pt, yshift=1.2pt}] (H1)--(I);
\draw[green, transform canvas={xshift=1pt, yshift=0pt}] (I)--(Id);
\draw[thick, ->-] (G1) -- node[midway, right]{$d$}(Gu); 
\draw[thick, ->-] (Gd) -- node[midway, right]{$b$}(G2);
\draw[thick, ->-] (G1) -- node[midway, above left]{$a$}(H1); \draw[thick, ->-] (G2) -- node[midway, below left]{$a$}(H2);
\draw[thick, ->-] (H1) -- node[midway, above]{$c$}(I); \draw[thick, ->-] (H2) -- node[midway, below]{$c$}(I);
\draw[thick, ->-] (I) -- node[midway, right]{$b$}(Id);
\draw[thick, ->-] (Iu) -- node[midway, right]{$c$}(I);
\fill (H1) circle (2pt);\fill (H2) circle (2pt);
\fill (G1) circle (2pt); \fill (G2) circle (2pt);
\fill (I) circle (2pt);
\coordinate (he111) at (13.29, 0.71); \coordinate (he211) at (14, 1); \coordinate (he311) at (14.71, 0.71); 
\coordinate (he411) at (13.29, -0.71); \coordinate (he511) at (14,-1); \coordinate (he611) at (14.71, -0.71);
\coordinate (he711) at (15.29, 0.71);
\draw[->-] (H1)--(he111); \draw[->-] (H1)--(he211); \draw[->-] (he311)--(H1);
\draw[->-] (he411)--(H2); \draw[->-] (H2)--(he511); \draw[->-] (H2)--(he611);
\draw[->-] (he711)--(I);
\node [above left] at (he111) {$b$}; \node [above] at (he211) {$c$}; \node [above] at (he311) {$d$};
\node [below left] at (he411) {$b$}; \node [below] at (he511) {$a$}; \node [below right] at (he611) {$d$};
\node[above] at (he711) {$d$};
\node[left] at (G1) {$v_0'$}; \node[left] at (G2) {$v_0''$};
\node[right] at (I) {$v_2$};

\end{tikzpicture}

\end{figure}
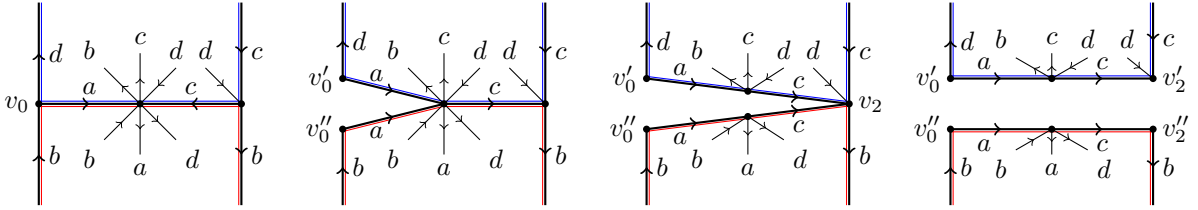

Henceforth, we will now be considering $\chi(X)$ \emph{after }the
piece unzipping procedure. Given any piece $P$, denote by $V_{P}$
its set of internal boundary vertices, so that 
\[
|V_{P}|=\mathfrak{e}(P)-1
\]
and 
\[
\sum_{P}\sum_{v\in V_{P}}\kappa(v)=\sum_{v\in V_{w}^{\mathrm{int}}}\kappa(v).
\]
Just as in (\ref{eq: curvature in terms of V-=00005Cgamma}), for
each piece $P,$ we have 
\begin{equation}
\sum_{v\in V_{P}}\kappa(v)\le\pi\left(\frac{2g-1}{2g}|V_{P}|-\sum_{v\in V_{P}}\frac{\mathfrak{he}(v)}{2g}\right).\label{eq: sum of internal curvature, degenerate, in terms of V_p}
\end{equation}
By Lemma \ref{lem: bound for e(P), he(P)}, we have 
\[
\begin{aligned}|V_{P}|+1 & \le(2g-1)\mathfrak{he}(P)+2g\\
 & =(2g-1)\sum_{v\in V_{P}}\mathfrak{he}(v)+2g.
\end{aligned}
\]
Hence, 
\[
|V_{P}|+(1-2g)\le(2g-1)\sum_{v\in V_{P}}\mathfrak{he}(v)
\]
and thus 
\begin{equation}
\frac{1}{2g(2g-1)}|V_{P}|-\frac{1}{2g}\le\sum_{v\in V_{P}}\frac{\mathfrak{he}(v)}{2g}.\label{eq: bound on size of V_p in terms of he(P)}
\end{equation}
Combining (\ref{eq: sum of internal curvature, degenerate, in terms of V_p})
with (\ref{eq: bound on size of V_p in terms of he(P)}) gives 
\[
\begin{aligned}\sum_{v\in V_{P}}\kappa(v) & \le\pi\left(\frac{2g-1}{2g}|V_{P}|-\frac{1}{2g(2g-1)}|V_{P}|+\frac{1}{2g}\right)\\
 & =\pi\left(\left(\frac{2g-2}{2g-1}\right)|V_{P}|+\frac{1}{2g}\right).
\end{aligned}
\]
Recall that the only pieces that positively contribute to $\sum_{v}\kappa(v)$
are those consisting only of WR--edges. Supposing there are $M$
such pieces in total, this gives 
\begin{equation}
\sum_{v\in V_{w}^{\mathrm{int}}}\kappa(v)\le\pi\left(\frac{2g-2}{2g-1}\right)\sum_{P}|V_{P}|+\frac{\pi M}{2g}.\label{eq: total internal curvature, degenerate}
\end{equation}
We also have the following lemma, which bounds the curvature contribution
from the external boundary vertices.
\begin{lem}
\label{lem:bound on curvature external vertices}We have 
\[
\sum_{v\in V_{w}^{\mathrm{ext}}}\kappa(v)\le\frac{-\pi M}{2g}.
\]
\end{lem}

\begin{proof}
Firstly, the external boundary vertices $v$ with $\chi\left(\link\right)=0$
are those with no WR--edges incident and only RRW--edges incident.
These necessarily have $|\mathcal{C}(v)|\ge4g$ and thus have $\kappa(v)\le0.$
At each remaining external boundary vertex, there are either $2$
incident WR--edges (or RRW--edges) or at least $4$ incident WR--edges
(or RRW--edges). If there are $2$, then the curvature of such a
vertex is at most $\pi\left(2-\chi\left(\link\right)\right)+\frac{-\pi}{2g}$
and if there are $\ge4$, then the curvature is at most $\pi\left(2-\chi\left(\link\right)\right)+\frac{-2\pi}{2g}.$
After the unzipping procedure, either $\chi\left(\link\right)\ge2$
or already has $\kappa(v)\le-\pi$. Thus, in any case, we can always
bound $\kappa(v)$ by either $\frac{-\pi}{2g}$ or $\frac{-2\pi}{2g}$,
depending on the number of incident WR--edges (or RRW--edges). Suppose
there are $M_{1}$ external boundary vertices with exactly $2$ incident
WR--edges (or RRW--edges) and $M_{2}$ external boundary vertices
with at least $4$ incident WR--edges (or RRW--edges). The total
number of incident WR--edges and RRW--edges at external boundary
vertices is clearly at least $2M,$ since this accounts for the number
of WR--incidences only (since every piece consisting of only WR--edges
has two endpoints), hence 
\[
2M_{1}+4M_{2}\ge2M.
\]
Therefore, the total curvature from external boundary vertices is
at most 
\[
\frac{-M_{1}\pi}{2g}+\frac{-2M_{2}\pi}{2g}\le\frac{-\pi M}{2g}.
\]
\end{proof}
Combining Lemma \ref{lem:bound on curvature external vertices} with
(\ref{eq: total internal curvature, degenerate}) yields the estimate
\[
\sum_{v\in V_{w}}\kappa(v)\le\pi\left(\frac{2g-2}{2g-1}\right)\sum_{P}|V_{P}|.
\]
It then remains only to bound $\sum_{P}|V_{P}|.$ This is the following
lemma, which completes our proof of Lemma \ref{lem:curvature of vertices},
and hence our proof of Proposition \ref{prop:bounding euler char of Gamma -k}
and thus of Theorem \ref{thm: bound on integral main}. 
\begin{lem}
We have 
\[
\sum_{P}|V_{P}|\le2(2g-1)k,
\]
with equality if and only if every piece $P$ satisfies $\mathfrak{e}(P)=(2g-1)\mathfrak{he}(P)+2g.$ 
\end{lem}

\begin{proof}
We observe that for each piece $P$, we have 
\begin{equation}
|V_{P}|=\mathfrak{e}(P)-1\le\frac{2g-1}{2g}\left(\mathfrak{e}(P)+\mathfrak{he}(P)\right).\label{eq: rearrangement of e(P), he(P) inequality}
\end{equation}
Indeed, this is simply a rearrangement of 
\[
\mathfrak{e}(P)\le(2g-1)\mathfrak{he}(P)+2g,
\]
which is true by Lemma \ref{lem: bound for e(P), he(P)}. Clearly,
equality holds in (\ref{eq: rearrangement of e(P), he(P) inequality})
if and only if equality holds in Lemma \ref{lem: bound for e(P), he(P)}
for the piece $P$. We then have 
\[
\sum_{P}\left(\mathfrak{e}(P)+\mathfrak{he}(P)\right)\le|E_{\mathrm{WG}}|+2|E_{\mathrm{WW}}|+2|E_{\mathrm{WWG}}|=4gk,
\]
and so 
\[
\begin{aligned}\sum_{P}|V_{P}| & \le\frac{2g-1}{2g}\sum_{P}\left(\mathfrak{e}(P)+\mathfrak{he}(P)\right)\\
 & \le2(2g-1)k,
\end{aligned}
\]
with equality if and only if $\mathfrak{e}(P)=(2g-1)\mathfrak{he}(P)+2g$
for every piece $P$. 
\end{proof}
There are several conditions that must be satisfied for any word $w$
representing $\gamma$ to have a value of $k$ and a $\left(\sigma_{x},\pi_{i}\right)\in\mathrm{Match}_{S_{k}}\left(\sigma_{x},\pi_{i}\right)$
such that $\chi\left(\samgamsig\right)=-k.$
\begin{itemize}
\item Since we use $|E_{\mathrm{WR}}|+2|E_{\mathrm{RR}}|+2|E_{\mathrm{RRW}}|=4gk$
and $\sum_{p}\mathfrak{e}(P)=|E_{\mathrm{WR}}|+|E_{\mathrm{RRW}}|,$
then there can be no RRW--edges. 
\item We would require every external boundary vertex $v$ to have $\kappa(v)=\frac{-\pi}{2g}$
of $\kappa(v)=\frac{-2\pi}{2g}$, depending on the number of incident
RW--edges. Therefore, there can be no RR--edges incident to any
external boundary vertex. 
\item We also use $\sum_{p}\mathfrak{he}(P)\le2|E_{\mathrm{RR}}|.$ This
inequality is strict unless every RR--edge is incident to a piece
at both endpoints (that is to say, there are no interior vertices). 
\item All boundary components consisting of WW--edges only must be paths
so as to contribute exactly $0$ to the curvature. 
\item All pieces in any putative gluing with $\chi\left(\samgamsig\right)=-k$
must satisfy either $\mathfrak{e}(P)=(2g-1)\mathfrak{he}(P)$ (if
there is a unique piece) or $\mathfrak{e}(P)=(2g-1)\mathfrak{he}(P)+2g$
(if there is more than one piece).
\end{itemize}
In light of the final point above, we can give the proof of Theorem
\ref{thm: -k-1 bound}. 
\begin{proof}[Proof of Theorem \ref{thm: -k-1 bound}]
 Clearly if (up to cyclic permutations) a word $w$ is the unique
shortest representative of the conjugacy class of $\gamma\in\Gamma_{g},$
then no subword of $w$ (or any of its cyclic permutations) can be
exactly half of the relator $R_{g}$ (of length $2g$). Thus, it is
not possible for a unique piece $P$ to satisfy $\mathfrak{e}(P)=(2g-1)\mathfrak{he}(P)$
or for any other pieces $P$ to satisfy $\mathfrak{e}(P)=(2g-1)\mathfrak{he}(P)+2g.$ 
\end{proof}
The above points give a very rigid structure for words $w$ that may
have such a gluing. Nevertheless, we finish with an example of a word
$w$ that is a shortest representative of the conjugacy class of some
$\gamma\in\Gamma_{g}$ for which there is a $k$ and a gluing $(\sigma_{x},\pi_{i})\in\mathrm{Match}_{S_{k}}\left(w,k\right)$
with $\chi\left(\samgamsig\right)=-k.$ This illustrates that, if
one wanted to improve on the bound in Theorem \ref{thm: bound on integral main}
using the methods of this paper, then there must be other considerations
made as to which shortest representative $w$ to choose. 
\begin{example}
\label{exa: -k example}Take $w=[a,b]d^{-1}ab[d,c]d^{-1}ab$ and $k=1.$
This has a gluing $(\sigma_{x},\pi_{i})\in\mathrm{Match}_{S_{k}}\left(w,k\right)$
with $\chi\left(\samgamsig\right)=-1$ as seen below. 

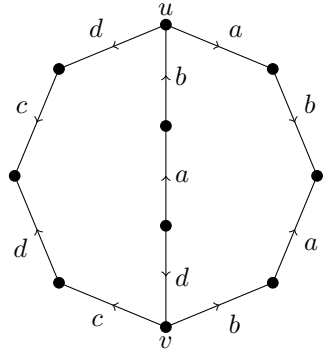
\begin{figure}[H]
\centering\begin{tikzpicture}[scale=1, baseline=(current bounding box.center)]

\coordinate (v1) at ( 0.0000,  2.0000);
\coordinate (v2) at ( 1.4142,  1.4142);
\coordinate (v3) at ( 2.0000,  0.0000);
\coordinate (v4) at ( 1.4142, -1.4142);
\coordinate (v5) at ( 0.0000, -2.0000);
\coordinate (v6) at (-1.4142, -1.4142);
\coordinate (v7) at (-2.0000,  0.0000);
\coordinate (v8) at (-1.4142,  1.4142);

\filldraw (v1) circle (2pt); \node[above] at (v1) {$u$};
\filldraw (v2) circle (2pt);
\filldraw (v3) circle (2pt);
\filldraw (v4) circle (2pt);
\filldraw (v5) circle (2pt); \node[below] at (v5) {$v$};
\filldraw (v6) circle (2pt);
\filldraw (v7) circle (2pt);
\filldraw (v8) circle (2pt);

\draw[->-] (v1) -- node[midway, above right]{$a$} (v2); 
\draw[->-] (v2) -- node[midway, above right]{$b$}(v3);
\draw[->-] (v4) -- node[midway, below right]{$a$} (v3);
\draw[->-] (v5) --  node[midway, below right]{$b$}(v4);
\draw[->-] (v5) -- node[midway, below left]{$c$}(v6);
\draw[->-] (v6) -- node[midway, below left]{$d$}(v7);
\draw[->-] (v8) -- node[midway, above left]{$c$}(v7);
\draw[->-] (v1) -- node[midway, above left]{$d$} (v8);

\coordinate (u1) at (0, -0.66); \filldraw (u1) circle(2pt);
\coordinate (u2) at (0, 0.66); \filldraw (u2) circle(2pt);

\draw[->-] (u1)-- node[midway, right]{$d$} (v5); 
\draw[->-] (u1)-- node[midway, right]{$a$} (u2); 
\draw[->-] (u2)-- node[midway, right]{$b$} (v1); 

\end{tikzpicture}

\caption{\label{fig: -k exactly gluing}This gluing for $k=1$ and $w=[a,b]d^{-1}ab[d,c]d^{-1}ab$
has two pieces, one reading $[a,b]$ glued clockwise and the other
reading $[d,c]$ glued anticlockwise. There are two external boundary
vertices labeled by $u$ and $v$ and the `$d^{-1}ab$' path through
the middle is made up of WW--edges. Clearly here $\chi=-1=-k.$ }
\end{figure}
\end{example}

\printbibliography

@article{Jones,
   author = {Vaughan Frederick Randal Jones},
   journal = {Subfactors: Proceedings of the Taniguchi Symposium on Operator Algebras},
   pages = {259--267},
   title = {The {P}otts model and the symmetric group},
   year = {1994},
}

@article{CollinsWeingartenShort,
   author = {Benoit Collins and Sho Matsumoto and Jonathan Novak},
   journal = {Notices of the American Math. Soc},
   month = {9},
   pages = {734--745},
   title = {The {W}eingarten Calculus},
   volume = {69},
   year = {2021},
}

@article{CollinsSniady2006,
   author = {Benoit Collins and Piotr Sniady},
   journal = {Comm.  Math. Phys.},
   month = {6},
   pages = {773--795},
   title = {Integration with Respect to the Haar Measure on Unitary, Orthogonal and Symplectic Group},
   volume = {264},
   year = {2006},
}

@article{Collins2003,
   author = {Benoit Collins},
   journal = {International Mathematics Research Notices},
   pages = {953--982},
   title = {Moments and cumulants of polynomial random variables on unitary groups, the {I}tzykson--{Z}uber integral, and free probability},
   volume = {2003},
   year = {2003},
}

@article{PuderParzanchewski2012,
   author = {Doron Puder and Ori Parzanchevski},
   journal = {Journal of the American Mathematical Society},
   month = {2},
   pages = {63--97},
   title = {Measure preserving words are primitive},
   volume = {28},
   year = {2015},
}

@article{Magee2021,
   author = {Michael Magee},
   journal = {Geometry and Toplology},
      title = {Random unitary representations of surface groups {II}: The large $n$ limit},
   year = {2025},
  pages = {1237--1281}
}

@article{Rota,
   author = {Gian-Carlo Rota},
   journal = {Zeitschrift f{\"u}r Wahrscheinlichkeitstheorie und Verwandte Gebiete},
   pages = {340--368},
   title = {On the foundations of combinatorial theory {I}. Theory of {M}{\"o}bius functions},
   volume = {2},
   year = {1964},
}

@article{HananyPuder,
   author = {Liam Hanany and Doron Puder},
   journal = {International Mathematics Research Notices},
   pages = {9221--9297},
   title = {Word measures on symmetric groups},
   volume = {11},
   year = {2023},
}

@article{Weingarten,
   author = {Donald Weingarten},
   journal = {J. Math. Phys.},
   pages = {999--1001},
   title = {Asymptotic behavior of group integrals in the limit of infinite rank},
   volume = {19},
   year = {1978},
}

@article{Xu,
   author = {Feng Xu},
   journal = {Communications in Mathematical Physics},
   pages = {287--307},
   title = {A random matrix model from two dimensional {Y}ang--{M}ills theory},
   volume = {190},
   year = {1997},
}

@article{Cassidy2023,
      title={Projection formulas and a refinement of {S}chur--{W}eyl--{J}ones duality for symmetric groups}, 
      author={Ewan Cassidy},
      year={2025},
      journal={preprint, arXiv 2312.01839} 
}

@article{Stallings,
  title={Topology of finite graphs},
  author={Stallings, John R},
  journal={Inventiones Mathematicae},
  volume={71},
  number={3},
  pages={551--565},
  year={1983},
  publisher={Springer-Verlag Berlin/Heidelberg}
}

@article{Nica,
author = {Nica, Alexandru},
title = {On the number of cycles of given length of a free word in several random permutations},
journal = {Random Structures \& Algorithms},
volume = {5},
number = {5},
pages = {703--730},
doi = {https://doi.org/10.1002/rsa.3240050506},
year = {1994}
}

@article{LinialPuder,
  title={Word maps and spectra of random graph lifts},
  author={Nathan Linial and Doron Puder},
  journal={Random Structures \& Algorithms},
  year={2010},
  volume={37},  
  pages={100--135}
}

@article{MageePuder2023, 
title={The Asymptotic Statistics of Random Covering Surfaces}, 
volume={11},
DOI={10.1017/fmp.2023.13}, 
journal={Forum of Mathematics, Pi}, 
author={Michael Magee and Doron Puder}, 
year={2023}, 
pages={e15}
}

@article{Magee22,
  title={Random unitary representations of surface groups {I}: Asymptotic expansions},
  author={Magee, Michael},
  journal={Communications in Mathematical Physics},
  volume={391},
  number={1},
  pages={119--171},
  year={2022},
  publisher={Springer}
}

@article{Cassidy2,
   title={Random permutations acting on {$k$}--tuples have near--optimal spectral gap for {$k=\mathrm{poly}(n)$}}, 
      author={Ewan Cassidy},
      year={2025},
      journal={preprint, arXiv 2412.13941} 
}

@inproceedings{SamSnowden,
  title={Stability patterns in representation theory},
  author={Sam, Steven V and Snowden, Andrew},
  booktitle={Forum of Mathematics, Sigma},
  volume={3},
  pages={e11},
  year={2015},
  organization={Cambridge University Press}
}

@article{MageePuderVanHandel,
  title={Strong convergence of uniformly random permutation representations of surface groups},
  author={Magee, Michael and Puder, Doron and van Handel, Ramon},
  journal={preprint, arXiv 2504.08988},
  year={2025}
}

@article{Etingof,
  title={Representation theory in complex rank, {I}},
  author={Etingof, Pavel},
  journal={Transformation Groups},
  volume={19},
  pages={359--381},
  year={2014},
  publisher={Springer}
}

@article{Dehn,
   title={Transformation der {K}urven auf zweiseitigen {F}l{\"a}chen},
   author={Max Dehn},
   journal={Mathematische Annalen},
  volume={72},
  number={3},
  pages={413--421},
  year={1912},
  publisher={Springer}
}

@article{BirmanSeries,
  title={Dehn's algorithm revisited, with applications to simple curves on surfaces},
  author={Birman, Joan and Series, Caroline},
  journal={Combinatorial group theory and topology (Alta, Utah, 1984)},
  volume={111},
  pages={451--478},
  year={1987}
}

@article{WiseAngles,
  title={Sectional curvature, compact cores, and local quasiconvexity},
  author={Wise, Daniel T},
  journal={Geometric \& Functional Analysis},
  volume={14},
  number={2},
  pages={433--468},
  year={2004},
  publisher={Springer}
}

@article{LiebeckShalev,
  title={Fuchsian groups, coverings of {R}iemann surfaces, subgroup growth, random quotients and random walks},
  author={Liebeck, Martin W and Shalev, Aner},
  journal={Journal of Algebra},
  volume={276},
  number={2},
  pages={552--601},
  year={2004},
  publisher={Elsevier}
}

@article{Hurwitz,
  title={{\"U}ber die {A}nzahl der {R}iemann'schen {F}l{\"a}chen mit gegebenen {V}erzweigungspunkten},
  author={Hurwitz, Adolf},
  journal={Mathematische Annalen},
  volume={55},
  number={1},
  pages={53--66},
  year={1901},
  publisher={Springer}
}

@article{Klukowski,
    author ={Adam Klukowski},
   title={In preparation},
   year={2026}
}

\noindent Ewan Cassidy, Department of Pure Mathematics and Mathematical
Statistics, University of Cambridge, Wilberforce Road, Cambridge,
CB3 0WB

\noindent egc45@cam.ac.uk
\end{document}